\documentclass[10pt,reqno]{amsart}
\usepackage{hyperref}
\usepackage{amssymb,amsmath, amsthm}
\usepackage{srcltx}
\usepackage{graphicx}
\usepackage[svgnames]{xcolor}

\usepackage[normalem]{ulem}

\newtheorem{thm}{Theorem}[section]
\newtheorem{lemma}[thm]{Lemma}

\newtheorem{prop}[thm]{Proposition}

\hypersetup{colorlinks=true,linkcolor=violet,citecolor=purple}

\theoremstyle{definition}
\newtheorem{defn}[thm]{Definition}
\newtheorem{rem}[thm]{Remark}

\makeatletter
\@addtoreset{equation}{section}

\makeatother

\newcommand{\PP}{\mathcal{P}}  
\newcommand{\w}{\omega}        
\newcommand{\ed}{\mathbf{d}}   
\newcommand{\J}{\mathbf{J}}    
\newcommand{\D}{\mathbf{D}}    
\newcommand{\Z}{\mathcal{Z}}  
\newcommand{\LieAlg}{\mathfrak{Lie}}   
\newcommand{\R}{\mathbb{R}}    
\newcommand{\C}{\mathbb{C}}    
\newcommand{\Ad}{\mathrm{Ad}}  
\newcommand{\ad}{\mathrm{ad}}  
\newcommand{\im}{\mathrm{im}\, }  
\newcommand{\OO}{\mathcal{O}}  
\newcommand{\g}{\mathfrak{g}}
\newcommand{\lil}{\mathfrak{l}}
\newcommand{\h}{\mathfrak{h}}
\newcommand{\n}{\mathfrak{n}}
\newcommand{\q}{\mathfrak{q}}
\newcommand{\m}{\mathfrak{m}}
\newcommand{\z}{\mathfrak{z}}
\newcommand{\be}{\begin{equation}}
\newcommand{\ee}{\end{equation}}
\newcommand{\bea}{\begin{eqnarray}}
\newcommand{\eea}{\end{eqnarray}}
\newcommand{\Proj}{\mathbb{P}} 
\newcommand{\bh}{\bar h}

\newcommand{\mgs}{G\times_{G_z}(\m^*\times N)} 
\newcommand{\Null}{\mathcal{N}}

\newcommand{\xx}{\mathbf{x}}
\newcommand{\zt}{\widetilde{z}}

\newcommand{\restr}[1]{\vrule height3ex width.4pt depth1.4ex\lower1.4ex\hbox{\scriptsize $\,#1$}}
\newcommand{\rrestr}[1]{\,\vrule height2ex width.4pt depth0.9ex\lower0.9ex\hbox{\scriptsize $\,#1$}}

\def\paragraph#1{\par\bigskip\noindent\textbf{#1}}

\begin{document}

\title[Hamiltonian relative equilibria]{Hamiltonian relative equilibria with continuous isotropy}

\author{James Montaldi}

\author{Miguel Rodr\'{\i}guez-Olmos}

\email{j.montaldi@manchester.ac.uk, miguel.rodriguez.olmos@upc.edu}

\address{JM: School of Mathematics, University of Manchester, Manchester M13 9PL, UK}
\address{MRO: Departamento de Matem\'atica Aplicada IV, Universidad Polit\'ecnica de Catalu\~na, Barcelona, Spain.}

\thanks{This research was partially supported by a Marie Curie Intra European Fellowship PIEF-GA-2008-220239 and a Marie Curie Reintegration grant PERG-GA-2010-27697. MR-O's research has been partially supported by Ministerio de Ciencia e Innovaci\'on (Spain), project MTM2011-22585 and Ministerio de Econom\'{\i}a y Competitividad (Spain), project MTM2014-54855-PÓ}


\maketitle

\begin{abstract}
In symmetric Hamiltonian systems, relative equilibria usually arise in continuous families. The geometry of these families in the setting of free actions of the symmetry group is well-understood.  Here we consider the question for non-free actions.  Some results are already known in this direction, and we use the so-called bundle equations to provide a systematic treatment of this question which both consolidates the known results, extending the scope of the results to deal with non-compact symmetry groups, as well as producing new results. Specifically we address questions about the
 stability, persistence and bifurcations of these relative equilibria.
 \end{abstract}

\tableofcontents

\section{Introduction}
This article deals with a particular kind of integral curve of
Hamiltonian flows on symplectic manifolds equivariant under the
canonical action of a Lie group. They are usually known in the literature as relative
equilibria, or steady state solutions,  a
terminology inherited from analytical mechanics. In this article we
follow an approach to the qualitative study of some local properties
of these solutions, based on the use of geometrically adapted
tubular neighbourhoods. This approach is an evolution of methods previously proposed in the literature by several authors (see \cite{RoSD97,RoWuLa02,Wu01}).

Let $(\PP,\omega)$ be a symplectic manifold, $G$  a Lie group acting
on  $\PP$ by symplectomorphisms and $h$  a
smooth $G$-invariant function on $\PP$. Hamilton's equations produce a
vector field $X_h$ on $\PP$, and due to the invariance of $h$ and the
symplectic structure under the $G$-action, the flow of this vector field sends group
orbits to group orbits. A relative equilibrium is an integral curve
of $X_h$ that belongs to a single group orbit.

The importance of relative equilibria is more apparent using orbit
reduction. The flow of $X_h$ naturally descends  to a continuous flow
on the quotient space $\PP/G$, and relative equilibria project to
fixed points of this reduced flow, therefore being (together with
periodic orbits, which will not be addressed in this article) the
primary object of qualitative local studies. If the action of $G$ on
$\PP$ is free and proper
the quotient space $\PP/G$ is a manifold and the
reduced flow can be shown to be smooth and Hamiltonian with respect
to some reduced Hamiltonian function and Poisson structure on the quotient space.
This means that relative equilibria can be studied via their
projected fixed points using many of the standard available techniques
from differential geometry and critical point theory on this
quotient space.

For these reasons, this article deals with equivariant
Hamiltonian flows for which the symmetry Lie group action has
singularities. In this situation this geometric approach is no
longer possible since, to start with, the quotient space $\PP/G$ is no
longer smooth. We focus on the problems of nonlinear stability of
relative equilibria, as well as on the organization of relative
equilibria in parametrized branches and their bifurcations. The study of these
topics in the presence of singularities has been going on for a number of
decades now, and a variety of methods have been used to attack
them, from the use of singular reduction and confinement arguments
\cite{OrRa99}, to the use of suitable linearizations of the
Hamiltonian field \cite{Lew92,LeSi98,ChLeOrRa03}
passing by topological methods \cite{Mon97,MoTo03,PaRoWu04}, to cite a few.
Analogous questions for free actions are addressed in a number of papers,
and in particular \cite{Pat95,Mon97,PaRo00}

The motivation for this article is twofold. On the one hand, we
wish to obtain a convenient framework adapted to the study of the
local dynamics of equivariant Hamiltonian flows in the presence of
singularities specifically adapted to the unique geometric features
of the problem, and powerful enough to serve as a unifying approach
to deal with all aspects related to this topic. For that, we have
adopted the ``bundle equations'' proposed in \cite{RoSD97} and
\cite{RoWuLa02} and have built on this formalism in order to obtain
a self contained geometric machinery especially adapted to deal with Lie
group actions with singularities. On the other hand, we show how
these ideas are implemented by proving within this framework some standard
results in the theory, with improved hypotheses, as well as obtaining several new relevant ones.

\paragraph{Organization.}
This article is organized as follows: Section \ref{sec general} collects some standard ge\-ne\-ra\-li\-ties on symmetric Hamiltonian systems and relative equilibria. Then the bundle equations, the main theoretical framework of this article, are introduced. We present a modification of the bundle equations of \cite{RoSD97} and \cite{RoWuLa02} which explicitly incorporates the existence of continuous stabilizers for the group action, resulting in the bundle equations with isotropy \eqref{genbundle1}, \eqref{genbundle2} and \eqref{genbundle3}. These are then used, in Proposition \ref{prop cond RE}, to obtain a system of differential equations locally characterizing relative equilibria for a given equivariant Hamitonian flow. One of the main objectives of this article is to find solutions of these equations, and for that we prove two technical results (Lemmas \ref{lemmaND} and \ref{lemmaD}) that will apply in the subsequent persistence and bifurcation results.

In Section \ref{sec stability} we use the bundle equations with isotropy in order to revisit the proof of well known stability results from \cite{LeSi98,OrRa99} and \cite{MoRo11}.  Our approach allows us to extend these to the statement in Theorem \ref{stability}.  Compared to \cite{LeSi98,OrRa99} we are able to avoid  making use of the ``orthogonality'' choice of velocity, and compared to \cite{MoRo11} we allow the possibility of non-compact momentum stabilizer $G_\mu$.  In Remark \ref{remarkstability} we  explicitly show with an example how our approach gives sharper stability results than those in \cite{LeSi98} and \cite{OrRa99}.

In Section \ref{persistence} we deal with the problem of persistence. That is, under which conditions a relative equilibrium persists to a continuous branch of relative equilibria. This differs with the problem of bifurcations, which considers as starting point a parametrized branch of relative equilibria and finds conditions for which at a particular value of the parameter, new relative equilibria exist arbitrarily close to the original branch. These new relative equilibria are typically organized in bifurcating branches.

We obtain four main results in this direction: Theorems \ref{thm persistence nondeg}, \ref{thm persistence same orbit type}, \ref{thm persistence deg} and \ref{thm formally stable}, based on different non-degeneracy or degeneracy hypotheses. Some of these results, or particular cases of them, had been obtained in the past using quite different approaches. Here we show how they all  follow in a straightforward way from the bundle equations with isotropy and are therefore unified under a common framework. Many of these results assume the group of symmetries to be compact, whereas we allow non-compact symmetry groups in all our results, thereby also providing minor extensions to the known results.

 Finally in Section \ref{sec bifurcations} we address the problem of finding bifurcations from pa\-ra\-me\-tri\-zed branches of relative equilibria. Using the notion of branches of relative equilibria of the same symplectic type introduced in Definition \ref{branchdef}, we are able to give, in Theorem \ref{thm bifurcations}, sufficient conditions under which one such branch exhibits bifurcations. This result is further elaborated in Theorem \ref{2dformallystable}, where we show that branches of formally stable relative equilibria satisfying a dimensionality condition  generically have a continuous range of bifurcation points.
 This property is specifically related to the existence of continuous stabilizers and shows that these bifurcations are  purely singular phenomena, since it is a well known result, originally due to Arnold \cite{Arn89}, that branches of nondegenerate regular relative equilibria cannot bifurcate. In Section \ref{sec example} the results of this article are illustrated in two  study cases, the sleeping Lagrange top and the system of two point vortices on the sphere.

We use a mixture of notation for derivatives: $\ed h$ denotes the differential of the function $h$, and if $h:N\times X\to\R$ is defined on a product of spaces, $\D_Nh$ denotes the partial derivative(s) of $h$ with respect to the $N$-variables. Occasionally, we write $\ed_z^2h$ to denote the Hessian matrix at the point $z$, but omit reference to the point $z$ if it is clear.

\section{Hamiltonian relative equilibria}\label{sec general}
In this section we provide necessary and sufficient conditions for
the existence of Hamiltonian relative equilibria. The study will be local, in a
neighbourhood of a given point in a symplectic phase space of a
symmetric Hamiltonian system. To this end, we will use the framework
provided by the bundle equations obtained in \cite{RoSD97} and
generalized in \cite{RoWuLa02} (see also Chapter 7 in
\cite{OrRa03}). In our approach we will additionally incorporate
in a critical way the freedom in several choices that exists if
the symmetry group action exhibits continuous isotropy groups.

To fix notation, we recall the standard definitions.
A proper $G$-Hamiltonian system is a quintuple
$(\PP,\w,G,\J,h)$, where $(\PP,\w)$ is a symplectic manifold,
$G$ is a Lie group acting properly and in a Hamiltonian fashion on $\PP$.  We write the action  by concatenation: $(g,z) \mapsto g\cdot z$.  The resulting momentum map
$\J:\PP\rightarrow \g^*$ satisfies
$$\iota_{\xi_\PP}\w=\ed\langle \J(\cdot),\xi\rangle,\quad \forall \xi\in \g.$$
Here $\xi_\PP(z)=\frac{d}{dt}\rrestr{t=0}\exp(t\xi)\cdot z$ is
the fundamental vector field associated to the element
$\xi\in\g$ (the Lie algebra of $G$) evaluated at $z\in\PP$.  We assume $\J$ to be $G$-equivariant with respect to the coadjoint action on $\g^*$:
$$\J(g\cdot z)=\Ad_{g^{-1}}^*(\J(z)), \quad \forall g\in G,\,z\in \PP,$$
where $\Ad^*:G\times\g^*\rightarrow\g^*$ is the coadjoint
representation of $G$.  See also Remark\,\ref{rmk:modified action} for the adaptation to coadjoint actions modified by a cocycle.
Finally,  $h\in C^\infty(\PP)$ is a $G$-invariant
Hamiltonian function on $\PP$.
The symplectic manifold $\PP$ is the phase space for the Hamiltonian
dynamical system given by the flow of the vector field $X_h$ defined
by Hamilton's equations
$$\iota_{X_h}\w=\ed h.$$

Since $h$ and $\w$ are $G$-invariant, the Hamiltonian vector field
$X_h$ is $G$-equivariant and so is its flow, therefore sending
group orbits to group orbits for all times. A well known
theorem by Noether states that $\J$ is constant along integral curves of $X_h$.
Our main object of study is defined next.
\begin{defn}\label{defn re}
A relative equilibrium is a point in $\PP$ whose integral curve lies in a group orbit.
\end{defn}

The following characterizations of relative equilibria are standard
and can be found, for instance in \cite{Mars92}.

\begin{prop}\label{charRE}
Let $(\PP,\w,G,\J,h)$ be a $G$-Hamiltonian system and $z\in \PP$ with momentum $\J(z)=\mu$. The following are equivalent.
\begin{itemize}
\item[(i)] $z$ is a relative equilibrium.
\item[(ii)] The group orbit $G\cdot z$ consists of relative equilibria. That is, if $z$ is a relative equilibrium, then
so is $g\cdot z$ for any $g$ in $G$.
\item[(iii)] There is an element $\xi\in \g_\mu$ such that $X_h(z)=\xi_\PP(z)$, where $\g_\mu\leq\g$ is the Lie algebra of the stabilizer of $\mu$ defined by $G_\mu=\{g\in
G\,:\,\Ad_{g^{-1}}^*\mu=\mu\}$.
\item[(iv)] There is an element $\xi\in \g_\mu$ such that $\bar z(t)=e^{t\xi}\cdot z$, where $\bar z(t)$ is the integral curve of $X_h$ with $\bar z(0)=z$.
\item[(v)] There is an element $\xi\in \g_\mu$ such that $z$ is a critical point of the augmented Hamiltonian
$$h_\xi:=h-\J^\xi,$$
with $\J^\xi(z):=\langle\J(z),\xi\rangle$.
\end{itemize}
\end{prop}

Here and throughout, we use $\leq$ to denote the relation of being a Lie subalgebra or a closed subgroup (according to context), and $\lhd$ to denote the relation of being a normal subgroup, or an ideal in a Lie algebra. We denote the normalizer of a subgroup $H$ in $G$ by $N_G(H)$.

\begin{rem}\label{remRE}
The element $\xi$ associated to $z$ is called a velocity for the
relative equilibrium. Notice that $\xi$ is not uniquely defined if
$z$ has continuous isotropy, and this is a key observation that will follow a long way. The isotropy group, or stabilizer, of $z$, $G_z=\{g\in
G\,:\,g\cdot z=z\}$ is a compact Lie subgroup of $G$ by the
properness hypothesis on the action. If $G_z$ has positive
dimension it has a non-trivial Lie algebra $\g_z$ whose elements $\eta$ satisfy by
definition $\eta_\PP(z)=0$.

Notice also that
the equivariance property of $\J$ implies in particular that
$\g_z\leq \g_\mu$. Therefore, by (iii) in Proposition
\ref{charRE} if $\xi$ is a velocity for $z$, so is $\xi+\eta$ for
any $\eta\in \g_z$. The converse is also true: any two admissible  velocities for $z$ must differ
by an element of $\g_z$. Also note that it follows again from
$(iii)$ and $(iv)$ in the same proposition that if $\xi$ is a velocity for $z$, then $\Ad_g\xi$ is a
velocity for the relative equilibrium $g\cdot z$.
\end{rem}

\begin{rem} \label{rmk:modified action}
It was stated earlier that we are assuming the momentum map to be equivariant with respect to the coadjoint action $\g^*$.  If $G$ is compact, this can always be arranged by an averaging argument  \cite{Mon97,OrRa03}.  However, for more general groups, $\J$ may not be able to be chosen to be equivariant in this sense.  In such cases, there is a cocycle $\theta:G\to\g^*$ (the `Souriau cocycle') for which the resulting modified action $\mathrm{Coad}^\theta_g\mu:=\Ad_{g^{-1}}^*\mu+\theta(g)$ renders the momentum map equivariant: $\J(g\cdot z) = \mathrm{Coad}^\theta_g\J(z)$.  Moreover, as shown in \cite{SchWu01} there is an analogous local model for the symplectic action of $G$ and similar bundle equations.
  It follows that all the results of this paper stated for the coadjoint action also hold in this more general case of a modified coadjoint action, provided the momentum isotropy group $G_\mu$ is understood to be relative to this modified action.
\end{rem}

We state a well-known algebraic lemma.

\begin{lemma} \label{lemma:algebraic}
Let $H$ be a compact subgroup of the Lie group $G$, and let $\g = \h\oplus \m$ be an $\Ad(H)$-invariant decomposition (as vector spaces).
\begin{itemize}
 \item[(i)] The subspace of fixed points $\g^H$ is a Lie subalgebra of $\g$, and the Lie bracket on $\g^H$ descends to one on $\m^H$.
 \item[(ii)] The Lie algebra of $N_G(H)$ satisfies $\LieAlg(N_G(H))=\h\oplus \m^H = \h+\g^H$. Consequently, if $L=N_G(H)/H$, and $\lil$ is  its Lie algebra, then $\lil\simeq\m^H$ as Lie algebras.
 \item[(iii)] If $H$ is a normal subgroup of $G$ then $\m^H=\m$.
\end{itemize}
\end{lemma}

In fact $\m^H=\m$ is equivalent to the (weaker) statement that $N_G(H)$ contains the connected component $G_0$ of $G$.

\begin{proof} The decomposition follows by using an $\Ad(H)$-invariant inner product, which exists because $H$ is compact, and defining $\m=\h^\perp$.

(i) It is easy to check that $\g^H$ is a Lie subalgebra of $\g$.  Now, $\h^H$ is in the centre of $\g^H$, and it follows that the Lie bracket descends to $\m^H\simeq \g^H/\h^H$.

(ii) For any $\xi\in\g$ write $\xi=\eta+\xi^\perp\in\h\oplus\m$.  Now $g\in N_G(H)$ means $gHg^{-1}=H$, and putting $g=\exp(t\xi)$ and differentiating at $t=0$ shows $\xi
\in\n:=\LieAlg(N_G(H))$ if and only if $\xi$ satisfies the (linear) condition
$$\Ad_h\xi-\xi\in\h,\quad\forall h\in H.$$
Writing $\xi=\eta+\xi^\perp\in\h\oplus\m$ (an $\Ad(H)$-invariant decomposition), this condition becomes
$$\Ad_h\xi^\perp-\xi^\perp=0,$$
which is equivalent to $\xi^\perp\in\m^H$, as required.

(iii) This follows immediately from the fact that $\n=\h+\m^H$.
\end{proof}

\subsection{The bundle equations with isotropy}\label{subsec bundle eqs}
\ \\
Since we are interested in the local properties of a Hamiltonian flow
in a neighbourhood of a relative equilibrium, we will substitute the phase space
$\PP$ by the symplectic tubular neighbourhood given by the
Marle-Guillemin-Sternberg (MGS) model. Originally due to Marle \cite{Mar85} and Guillemin and Sternberg \cite{GuSt84}, it is now standard and details can be found for
instance in \cite{OrRa03}.  We briefly recall its construction.

Let $z\in \PP$ be an arbitrary point with
momentum $\J(z)=\mu$. Let $\g_z$ and $\g_\mu$ be the Lie algebras of
the stabilizers of $z$ and $\mu$ as before. We will assume through
this paper that $\mu$ is a split element of $\g^*$, meaning that
there is a $G_\mu$-coadjoint invariant splitting
$\g=\g_\mu\oplus\q$.
Examples of split momentum values are the cases of compact or
Abelian $G_\mu$. For an example of non-split momentum see
\cite{RoWuLa02}.

Since by the equivariance of $\J$ we have $G_z\leq
G_\mu$ (we use $\leq$ to mean `is a closed subgroup of'), and since $G_z$ is compact, we can find a $G_z$-invariant
splitting
\begin{equation}\label{gmusplitting}\g_\mu=\m\oplus\g_z,\end{equation} with associated dual invariant splitting
$\g_\mu^*=\m^*\oplus\g_z^*$. Here $\g_z^*$ and $\m^*$ are identified
with the annihilators of $\m$ and $\g_z$ in $\g_\mu^*$ respectively.
Finally, combining these splittings of the Lie algebras, we
have a $G_z$-invariant splitting $\g=\m\oplus\g_z\oplus \q$. For
each of these splittings we will denote by $\Proj_\lil$ the
projection onto the factor $\lil$.

Next, choose $N=N_z$ to be a  complement to $\g_\mu\cdot z$ in $\ker
T_z\J$, $G_z$-invariant with respect to the induced linear
representation of $G_z$ on $T_z\PP$. It follows from the general
properties of Witt-Artin decompositions for Hamiltonian actions that
such a choice is always possible. Moreover, $(N,\Omega)$ is a symplectic vector space, where $\Omega$ is the
restriction to $N$ of $\w(z)$,  and the
linear action of $G_z$ on $N$ is Hamiltonian, with associated equivariant
momentum map $\J_N:N\rightarrow \g_z^*$ defined by
\begin{equation}\label{JN}\langle \J_N(v),\eta\rangle=\frac 12\,\Omega(\eta\cdot v,v),\quad \forall \eta\in \g_z,\,v\in N,\end{equation}
where $\eta\cdot v=\frac{d}{dt}\rrestr{t=0}e^{t\eta}\cdot v$ is the infinitesimal generator at $v$ corresponding to $\eta\in\g_z$.  The space $N$ is usually called the symplectic normal space (or symplectic slice) at $z$ for the $G$-action on $\PP$.

The product space $G\times \m^*\times N$ supports free actions of both $G$ and $G_z$, given by
\begin{eqnarray}
\label{GYaction} g'\cdot (g,\rho,v) & = & (g'g,\rho,v)\quad\quad g'\in G\\
\label{GxYaction} h\cdot (g,\rho,v) & = & (gh^{-1},\Ad_{h^{-1}}^*\rho,h\cdot v)\quad\quad h\in G_z
\end{eqnarray}
for all $g\in G,\,\rho\in\m^*$ and $v\in N$. It is clear that these
actions are free and commute. Consider the principal bundle
associated to the $G_z$-action
\begin{eqnarray*}\pi:G\times \m^*\times N & \rightarrow & G\times_{G_z}(\m^*\times N)\\
(g,\rho,v) & \mapsto & [g,\rho,v],
\end{eqnarray*}
where $[g,\rho,v]$ is the equivalence class of points consisting of the orbit $G_z\cdot (g,\rho,v)\subset G\times \m^*\times N$. The $G$-action \eqref{GYaction} on $G\times \m^*\times N$ descends to a smooth action on $\mgs$ given by
\begin{equation}\label{tubeaction}g'\cdot[g,\rho,v]=[g'g,\rho,v]\quad\quad g'\in G.\end{equation}
It follows from the MGS construction that there
are open neighbourhoods  of the origin in $\m^*$ and $N$
for which it is possible to define a local symplectic form on
 $\mgs$, as well as a local $G$-equivariant
symplectomorphism
$$\varphi:\mgs\rightarrow \PP$$
onto an open $G$-invariant neighbourhood of the orbit $G\cdot z$, and $\varphi$ satisfies
$\varphi([e,0,0])=z$. We will denote by $Y$ the domain of this
diffeomorphism in $\mgs$. The concrete expression of the mentioned symplectic form $\w_Y$ can be given as
follows: Every tangent vector to $Y$ can be written as
$\gamma_{\lambda,\dot\rho,\dot v}([g,\rho,v])\in T_{[g,\rho,v]}Y$
with
$$\gamma_{\lambda,\dot\rho,\dot v}([g,\rho,v])=T_{(g,\rho,v)}\pi(g\cdot\lambda,\dot\rho,\dot v),$$
where $\lambda\in\g, \dot\rho\in \m^*$ and $\dot v\in N$. Here,
$g\cdot \lambda$ is the concatenation notation for $T_eL_g\lambda
\in T_gG$, where $L:G\times G\rightarrow G$ is the left translation
on $G$, given by $L_{g'}g=g'g$. We then have

\begin{equation}
\label{wY}
\begin{array}{ccc}  \w_Y(\gamma_{\lambda_1,\dot\rho_1,\dot v_1}([g,\rho,v]),\gamma_{\lambda_2,\dot\rho_2,\dot v_2}([g,\rho,v]))  \hskip-3cm  \\[4pt]
 &= &    \langle\dot\rho_2+T_v\J_N(\dot v_2),\lambda_1\rangle-\langle\dot\rho_1+T_v\J_N(\dot v_1),\lambda_2\rangle\\
 && +\langle\mu+\rho+\J_N(v),[\lambda_1,\lambda_2]\rangle+\Omega(\dot
v_1,\dot v_2).
\end{array}
\end{equation}

In addition, the induced $G$-action on $(Y,\w_Y)$ given by \eqref{tubeaction} is Hamiltonian with a
locally defined equivariant momentum map
$\J_Y=\J\circ\varphi:Y\rightarrow \g^*$. The explicit expression for
$\J_Y$ is very simple and provides in this local model  a normal
form for $\J$. It is given by
\begin{equation}\label{JY}\J_Y([g,\rho,v])=\Ad_{g^{-1}}^*(\mu+\rho+\J_N(v)).\end{equation}

The idea exploited in \cite{RoSD97} and \cite{RoWuLa02} consists in taking the pullback
by $\varphi$ of the Hamiltonian vector field $X_h$ to $Y$ and then
lifting it to $G\times \m^*\times N$. Then the
differential equations for the flow of a choice of lifted vector
field are obtained, providing a general framework to study the local
dynamics of $X_h$ near $G\cdot z$.

In order to state these equations, note first that $h\circ \varphi$
is a $G$-invariant function on $Y$, so it lifts to a $G\times
G_z$-invariant function $h\circ\varphi\circ\pi$ on
$G\times\m^*\times N$. From equations \eqref{GYaction} and
\eqref{GxYaction} it follows that
\begin{eqnarray*}(h\circ\varphi\circ\pi)(g,\rho,v) & = & (\varphi\circ\pi)^*h(g\cdot (e,\rho,v))\ =\  (\varphi\circ\pi)^*h(e,\rho,v)\\
(h\circ\pi\circ\varphi)(l\cdot(g,\rho,v)) & = & (\varphi\circ\pi)^*h(gl^{-1},\Ad_{l^{-1}}^*\rho,l\cdot v)=(\varphi\circ\pi)^*h(e,\Ad_{l^{-1}}^*\rho,l\cdot v),
\end{eqnarray*}
for all $g\in G,\,l\in G_z,\,\rho\in\m^*$ and $v\in N$.
Therefore we can identify $h\circ\varphi\circ\pi$ with a $G_z$-invariant function on $\m^*\times N$ that we will denote by
\begin{equation}\label{eq:h-bar}
\bh\in C^\infty(\m^*\times N)^{G_z}.
\end{equation}

Since $\pi$ is a locally trivial fibration, a $\pi$-projectable
local vector field $X\in\mathfrak{X}(G\times \m^*\times N)$ can be
expressed as
$$X(g,\rho,v)=(g\cdot(X_{\g}(g,\rho,v)),X_{\m^*}(g,\rho,v),X_{N}(g,\rho,v)),$$
where $X_{\g}(g,\rho,v)\in\g,\,X_{\m^*}(g,\rho,v)\in\m^*$, and
$X_{N}(g,\rho,v)\in N$. The bundle equations on $G\times\m^*\times
N$ define a vector field on $G\times\m^*\times N$ whose components are given by:
\begin{eqnarray*}
X_{\g}(g,\rho,v) & = &  D_{\m^*}\bh(\rho,v)\\
X_{\m^*}(g,\rho,v) & = & \Proj_{\m^*}\left(\ad^*_{D_{\m^*}\bh(\rho,v)}(\rho+\J_N(v))\right)\\
X_N(g,\rho,v) & = & \Omega^\sharp(D_N\bh(\rho,v)).
\end{eqnarray*}
Here $D_{\m^*}\bh$ and $D_N\bh$ denote the partial derivatives of
$h\in C^\infty(\m^*\times N)^{G_z}$ with respect to $\m^*$ and $N$
respectively, and $\Omega^\sharp:N^*\rightarrow N$ is the linear
$G_z$-equivariant isomorphism induced from $\Omega$. As shown in
\cite{RoSD97,RoWuLa02,OrRa03}, these are the differential equations
for the flow of the unique $\pi$-projectable local vector field $X$
on $G\times \m^*\times N$ which is a lift of
$\varphi^*X_h\in\mathfrak{X}(G\times_{G_z}(\m^*\times N))$ (i.e.
such that $X$ and $\varphi^*X_h$ are $\pi$-related$)$ and satisfies
the additional condition \begin{equation}\label{choicelift}\Proj_{\g_z}(X_\g)=0.\end{equation}
Notice that since $D_{\m^*}\bh(\rho,v)\in \m$, the above equations imply that $\Proj_{\q}(X_\g)=0$, which is a consequence of the condition on $\mu$ to be split, and it is not related to the condition \eqref{choicelift}. Notice also that since $\varphi$ is a $G$-equivariant symplectomorphism, $\varphi^*X_h$ is actually $X_{h\circ \varphi}$, the Hamiltonian vector field corresponding to the local $G$-Hamiltonian system $(Y,\w_Y,G,\J_Y,h\circ \varphi)$, which is a local model for $(\PP,\w,G,\J,h)$.

However, in order to obtain all the projectable vector fields on $G\times \m^*\times N$ that
are $\pi$-related to $\varphi^*X_h$ one needs to include all vertical vector
fields tangent to the $\pi$-fibres. These are, in view of
\eqref{GxYaction}, of the  form
$$X(g,\rho,v)=(g\cdot\eta,\ad_\eta^*\rho,-\eta\cdot v),$$
where $\eta:G\times \m^*\times N\rightarrow \g_z$ is an arbitrary
smooth map.
In the following we will fix $\eta$ to be an arbitrary
constant in $\g_z$, since the vector fields obtained in this way
generate the module of vertical vector fields.

Using the facts that $\ad^*_\eta\rho\in\m^*$ (since $\m^*$ is
$G_{z}$-invariant), $\ad^*_\eta(\J_N(v))\in \g_z^*$ (since
$\J_N(v)\in \g_z^*$), and $\eta\cdot v=\Omega^\sharp(\ed\langle
\J_N(\cdot),\eta\rangle)$ (since the $G_z$-action on $N$ is Hamiltonian), we can write the equations of the flow of
the most general local vector field on $G\times \m^*\times N$ that
projects to $\varphi^*X_h$ as
\begin{eqnarray}
\label{genbundle1}\dot g & = & g\cdot(D_{\m^*}\bh(\rho,v)+\eta)\\
\label{genbundle2}\dot \rho & = & \Proj_{\m^*}\left(\ad^*_{D_{\m^*}\bh(\rho,v)+\eta}(\rho+\J_N(v))\right)\\
\label{genbundle3}\dot v & = & \Omega^\sharp(D_N\bh(\rho,v)-\ed\J_N^\eta(v)).
\end{eqnarray}
where $\eta\in\g_z$ is arbitrary and the function $\J_N^\eta$ is defined by $\J_N^\eta(v):=\langle\J_N(v),\eta\rangle$ for all $\eta\in\g_z$ and $v\in N$.

We also recall for future reference the following identity, which is
satisfied for integral curves of the above flow (see equation
(7.7.9) in \cite{OrRa03}).
\begin{equation}\label{impliedcond}\Proj_{\g_z^*}\left(\ad^*_{D_{\m^*}\bh(\rho,v)+\eta}\left(\rho+\J_N(v)\right)\right)=T_v\J_N(\dot v).\end{equation}

\subsection{Characterization of relative equilibria}\label{subsec charac}
\ \\
We can now reduce the problem of finding relative equilibria of the
$G$-Hamiltonian system $(\PP,\w,G,\J,h)$ near $z\in\PP$ to find
solutions of
\eqref{genbundle1},\eqref{genbundle2},\eqref{genbundle3} that
project to relative equilibria of $X_{h\circ\varphi}$ near $[e,0,0]$
in $Y$. We then have the following characterization.

\begin{prop}\label{prop cond RE}
A point $[g,\rho,v]$ near $z=[e,0,0]\in Y$ is a relative equilibrium
for $X_{h\circ\varphi}$ if and only if there exists an element
$\eta\in\g_z$ such that
\begin{eqnarray}
\label{re1} \ad^*_{D_{\m^*}\bh(\rho,v)+\eta}(\rho+\J_N(v)) & = & 0, \quad \text{and}\\
\label{re2} D_N\bh_\eta(\rho,v) & = & 0,
\end{eqnarray}
or equivalently
\begin{eqnarray}
\label{req1}\Proj_{\m^*}\left( \ad^*_{D_{\m^*}\bh(\rho,v)+\eta}(\rho+\J_N(v))\right) & = & 0, \quad \text{and}\\
\label{req2} D_N\bh_\eta(\rho,v) & = & 0,
\end{eqnarray}
where
$$\bh_\eta(\rho,v):=\bh(\rho,v)-\J_N^\eta(v).$$
The element of $\g$ given by $\Ad_{g}(D_{\m^*}h(\rho,v)+\eta)$ is a velocity for the relative equilibrium $[g,\rho,v]$.
\end{prop}

\begin{proof}
Let $[g,\rho,v]$ be in the domain $Y$ of the diffeomorphism $\varphi$. First, by Remark \ref{remRE}, $[g,\rho,v]$ is a relative equilibrium if and only if $[e,\rho,v]$ also is, so we can restrict the study to points of this form. Let $[e,\rho,v](t)$ be the integral curve of $\varphi^*X_h$ with initial condition $[e,\rho,v](0)=[e,\rho,v]$. Therefore, we can find representative curves $g(t),\rho(t),v(t)$ with initial conditions $g(0)=e,\rho(0)=\rho,v(0)=v$, respectively, such that $[e,\rho,v](t)=[g(t),\rho(t),v(t)]$. According to Proposition \ref{charRE}, $[e,\rho,v]$ is a relative equilibrium for the local $G$-Hamiltonian system on $\mgs$ corresponding to the Hamiltonian vector field $\varphi^*X_h=X_{h\circ \varphi}$, with velocity $\xi\in \g$ if and only if
$$\frac{d}{dt}\restr{t=0}[e,\rho,v](t)=\xi_Y([e,\rho,v]).$$
This is equivalent, in view of \eqref{GYaction} to
\begin{eqnarray*} \dot g(0) & = & \xi+\eta'\\
\dot\rho (0) & = & \ad^*_{\eta'}\rho\\
\dot v (0) & = & -\eta'\cdot v.
\end{eqnarray*}
for some $\eta'\in\g_z$. Since in the equations \eqref{genbundle1},\eqref{genbundle2},\eqref{genbundle3}, $\eta$ is arbitrary, we can absorb $\eta'$ into $\eta$ so that  the relative equilibrium conditions are
\begin{eqnarray}
\xi & = & D_{\m^*}\bh(\rho,v)+\eta\nonumber\\
\label{proof1}0 & = & \Proj_{\m^*}\left(\ad^*_{D_{\m^*}\bh(\rho,v)+\eta}(\rho+\J_N(v))\right)\\
0 & = & D_N\bh_\eta(\rho,v) \nonumber
\end{eqnarray}
Note that using \eqref{genbundle1},\eqref{genbundle2},\eqref{genbundle3} with this choice of isotropy $\eta$, we have $\dot v(0)=0$. Therefore, from \eqref{impliedcond} it follows that $\Proj_{\g_z^*}\left(\ad^*_{D_{\m^*}\bh(\rho,v)+\eta}(\rho+\J_N(v))\right)=0$. This, together with \eqref{proof1} is equivalent to $\ad^*_{D_{\m^*}\bh(\rho,v)+\eta}(\rho+\J_N(v))=0$, since $D_{\m^*}\bh(\rho,v)+\eta\in\g_\mu$, $\rho+\J_N(v)\in\g_\mu^*$ and $\g_\mu^*=\g_z^*\oplus\m^*$. Therefore we have obtained that $[e,\rho,v]$ is a relative equilibrium on $G\times_{G_z}(\m^*\times N)$ near $[e,0,0]$ with velocity $\xi$ if and only if there exists $\eta\in \g_z$ such that $\xi=D_{\m^*}\bh(\rho,v)+\eta$ and $D_N\bh_\eta(\rho,v)=0$. And by Remark \ref{remRE}, this is equivalent to $[g,\rho,v]$ being a relative equilibrium with velocity $\Ad_g(D_{\m^*}\bh(\rho,v)+\eta)$.
\end{proof}

\begin{rem}
Note that \eqref{genbundle2} implies that the $\m^*$-components of the vector fields on $G\times\m^*\times N$ that are $\pi$-related to $X_{h\circ\varphi}$ are always tangent to  the coadjoint orbit of $G_\mu$ which contains $\rho+\J_N(v)$. This equation can be studied by topological methods, as in \cite{Mon97}. It can also be simplified by imposing conditions on $G_\mu$. For instance, if $G_\mu$ is Abelian, \eqref{re1} is automatically satisfied. Also, the study of some dynamical properties can be simplified if these coadjoint orbits are bounded, as we shall see in Theorem \ref{stability}. On the other hand, equation \eqref{re2} is a critical point equation for a family of  functions on $N$ parametrized by $(\eta,\rho)\in\g_z\times \m^*$. Therefore, \eqref{re2} is naturally suited to be studied by methods based on the implicit function theorem or singularity theory. Note also that in general $\bh_\eta\in C^\infty(\m^*\times N)$ is no longer $G_z$-invariant, due to the presence of $\eta$, but only $(G_z)_\eta$-invariant.
\end{rem}

\subsection{Two lemmas}\label{subsec lemmas}
\ \\
We will now focus on the solutions of equation \eqref{re2}
 under the initial assumption that
$0\in N$ is a critical point of $\bh_\eta(0,\cdot)$ for some
$\eta\in\g_z$.  Notice that the dependence of $\bh_\eta$ on
$\eta$ is only through the term $\J_N^\eta$ and this is a quadratic
polynomial on $N$. It follows that if $0$ is a critical point of
$\bh_\eta(\rho,\cdot)$ for some $\eta\in\g_z$ then $0$ is also a
critical point of $\bh_{\eta'}(\rho,\cdot)$ for any other element in
$\eta'\in\g_z$, since $0$ is always a critical point of
$\J^{\eta'}_N$. The problem consists then in finding triples
$(\rho,\eta,v)\in \m^*\times \g_z\times N$ satisfying
$$D_N\bh_\eta(\rho,v)=0,$$
under the initial assumption $$D_N\bh_{\eta'}(0,0)=0\quad \text{for
every} \quad \eta'\in\g_z.$$ In order to attack this problem,
we state two main lemmas, based respectively on non-degeneracy
and degeneracy properties of suitable restrictions of a function
on $\g_z\times\m^*\times N$. These two lemmas will be used in most
of the  results of this paper.

Recall that if $G$ acts on a set $X$, and $K$ is a subgroup of $G$ then one writes
$$X^K=\{x\in X\,:\,k\cdot x=x\quad\text{for all}\quad k\in K\}$$
for the fixed point set. Moreover, the action of $G$ on $X$ restricts to an action of the normalizer $N_G(K)$ on $X^K$, and hence defines an action on $X^K$ of $N_G(K)/K$.   Before stating the lemmas, we recall the Principle of Symmetric Criticality, due to Palais \cite{Pal79}, which we will have recourse to a number of times.

\medskip

\noindent\textbf{Principle of Symmetric Criticality} \emph{Suppose a Lie group $G$ acts properly on a manifold $M$, and let $f:M\to\R$ be a smooth $G$-invariant function.  Let $x\in M^H$. Then $\ed f(x)\in (T_x^*M)^H$, which as $H$ is compact can be identified with $T_x^*(M^H)$.  In particular, it follows that if $x$ is a critical point of $f\rrestr{M^H}$ then it is a critical point of $f$.}

\medskip
We now have the first of the technical lemmas.

\begin{lemma}[\bf non-degeneracies]\label{lemmaND}
Suppose that $H\leq G_z$ and let $f\in C^\infty(\g_z\times\m^*\times N)^H$ be a smooth $H$-invariant function.
Suppose there is a subgroup $K\leq H$,  satisfying
$$
\begin{array}{l}D_Nf(0,0,0)=0,\quad \text{and} \\[4pt]
D^2_{N}f(0,0,0)\rrestr{N^{K}}\quad\text{is non-degenerate}.
\end{array}$$
Then, there is a unique local (defined in a neighbourhood of $(0,0)$) smooth $N_{H}(K)$-equivariant map
$v:{\m^*}^{K}\times\g_z^{K}\rightarrow N^{K}$ satisfying $v(0,0)=0$
and
$$D_Nf(\eta',\rho,v(\rho,\eta'))=0$$
for every $(\rho,\eta')\in {\m^*}^{K}\times\g_z^{K}$ near $(0,0)$.

 The stabilizer of
$m=(\rho,v(\rho,\eta'))\in{\m^*}^K\times N^K$ is
$(G_z)_m=(G_z)_\rho\cap (G_z)_{v(\rho,\eta')}$ and satisfies
$$K\leq (G_z)_{m}\leq (G_z)_\rho.$$
\end{lemma}
\begin{proof}

 Notice that ${\m^*}^K$, $\g_z^K$
and $N^K$ support linear actions of $N_{H}(K)$, which is the maximal
subgroup of ${H}$ that leaves invariant these subspaces. By the
pull-back property of Hessians, it follows that
$D^2_{N}f(0,0,0)\rrestr{N^{K}}=D^2_{N^K}(f\rrestr{N^{K}})(0,0,0)$. If this bilinear
form is non-degenerate, it follows from the Implicit Function Theorem
that there is a unique map $v:{\m^*}^K\times\g_z^K\rightarrow N^K$
defined in a neighbourhood of $(0,0)$ such that
$$D_{N^K}(f\rrestr{N^{K}})((\eta',\rho,v(\rho,\eta'))=0$$
for every $(\rho,\eta')$ in ${\m^*}^K\times\g_z^K$, and these are
all the points in ${\m^*}^K\times\g_z^K\times N^K$ satisfying the  equation $D_{N^K}(f\rrestr{N^{K}})(\eta',\rho,v)=0$ near $(0,0,0)$. By the invariance
properties of $f\rrestr{N^{K}}$, it follows that the map $v$ is
$N_{H}(K)$-equivariant.

Since $f(\eta',\rho,\cdot)\in C^\infty(N)$ is in particular
$K$-invariant for any $(\rho,\eta')\in {\m^*}^K\times\g_z^K$, it
follows from the Principle of Symmetric Criticality, and
$v(\rho,\eta')\in N^K$, that
$$D_{N}f(\eta',\rho,v(\rho,\eta'))=0.$$
The property about the stabilizer of $(\rho,v(\rho,\eta'))$ is
obvious.
\end{proof}

The second technical lemma involves degenerate critical points, and will be applied to finding bifurcating branches of relative equilibria.  It is usual to apply methods of singularity theory (e.g., \cite{PoSt,Mon-SBC}) to study the nature of critical points when a function with a degenerate critical point is deformed. However, to avoid assumptions on higher order terms, and discussions of finite determinacy and versal unfoldings, we consider here cases that can be reduced to one variable, where one can give a general existence result.

An action of a compact Lie group $K$ on a manifold $M$ is said to be of \emph{cohomogeneity one} if the orbit space $M/K$ is 1-dimensional.  Note that any representation of cohomogeneity one is necessarily irreducible.  Some simple examples are the representations of the trivial group on $\R$, $\mathrm{SO}(n)$ on $\R^n$ and $\mathrm{SU}(n)$ on $\R^{2n}\simeq\C^n$.

Let $\sigma(u)$ be a continuous real-valued function on a topological space. We say $\sigma$  \emph{crosses 0 at} $u_0$ if $\sigma(u_0)=0$ and in any neighbourhood of $u_0$, there exist $u_1,u_2$ such that $\sigma(u_1)>0$ and $\sigma(u_2)<0$.

\begin{lemma}[\bf degeneracies]\label{lemmaD}
Let $W$ and $N$ be representations of a compact Lie group $K$, and let $f\in C^\infty(W\times N)^K$ be a smooth $K$-invariant function.  Let $L\leq K$, and $\Lambda$ a path connected open neighbourhood of the origin in $W^L$. Suppose that
$$D_N f(\lambda,0)=0\quad \text{for all}\;\lambda\in \Lambda.$$
Let $\Null:=\ker D^2_N f(0,0)\cap N^{L}$. Suppose $\Null\neq0$ and that the following two conditions are fulfilled:
\begin{itemize}
\item[(i)]  the representation of $N_K(L)$ on $\Null$ is of cohomogeneity one (and hence irreducible), and
\item[(ii)] suppose that $\sigma(\lambda)$ crosses 0 at $\lambda=0$, where $\sigma(\lambda)$ is the eigenvalue of $D^2_N f(\lambda,0)\rrestr{\Null}$ (thus $\sigma(0)=0$).
\end{itemize}
Let $N^L=\Null\oplus S$ be an $N_K(L)$-invariant decomposition. Then for each sufficiently small  $v\in \Null$ there is a $\lambda=\lambda_v\in\Lambda$ and $s=s_v\in S$ (not necessarily unique) such that
$D_Nf(\lambda_v,v, s_v)=0$.

The stabilizer of the critical point $m=(\lambda_v,v,s_v)$ is $K_m=K_{\lambda_v}\cap K_v$
which satisfies $L\leq K_m\leq N_K(L)$.
\end{lemma}

\begin{proof}
Notice that by hypothesis, for any fixed $\lambda\in \Lambda$ the function $f(\lambda,\cdot)\in C^\infty(N)$ is $L$-invariant.
Therefore, according to the Principle of Symmetric Criticality, a point $v\in N^L$ is a critical point of
$f(\lambda,\cdot)$  if and only if the restriction of $f(\lambda,\cdot)$ to $N^L$, denoted by $f^L(\lambda,\cdot)$, has a critical point at $v$.
This is equivalent to $D_Nf(\lambda,v)\rrestr{N^L}=0$.
Notice also that the pull-back property for Hessians implies
$$D^2_N f(\lambda,v)\rrestr{N^L}=D^2_{N^L}f^L(\lambda,v).$$
Consequently, if (i) and (ii) are satisfied, then the (unique) eigenvalue $\sigma(\lambda)$ of
$D^2_{N^L}f^L(\lambda,0)\rrestr{\Null}$ is the only eigenvalue of $D^2_{N^L}f^L(\lambda,0)$ that
crosses $0$ at $\lambda=0$, the others being bounded away from zero.
We now apply the Splitting Lemma from singularity theory,
see for example \cite{PoSt,Mon-SBC}.

Let $N^L=\Null\oplus S$ be the given splitting.  Then $D^2_Nf(0,0)\rrestr{S}$ is non-degenerate.  Writing $v=(y,s)$ with $y\in\Null,\;s\in S$, there is an $N_K(L)$-equivariant change of coordinates of the form $(\lambda,y,\bar s) \mapsto (\lambda, y, s(\lambda,y,\bar s))$, such that
$$f(\lambda,y,s(\lambda,y,\bar s)) = Q(\bar s) + g(\lambda,y).$$
Here $Q$ is the $N_K(L)$-invariant non-degenerate quadratic form $\frac12D^2f\rrestr{S}$, and $g$ is an $N_K(L)$-invariant function for which $D_{\Null}g(\lambda,0)=0$ and  $D^2_\Null g(0,0)= 0$.
It follows that $f(\lambda,\cdot)$ has a critical point at $v=(y,s)\in N^L$ if and only if $\bar s(\lambda,y,s)=0$ and $g(\lambda,\cdot)$ has a critical point at $y$.

Since we are only considering critical points with respect to $N^L$, we can replace $f(\lambda,y,s)$ by $f(\lambda,y,s)-f(\lambda,0,0)$ and thus, without loss of generality,  suppose that $f(\lambda,0,0)=0$. And hence $g(\lambda,0)=0$.

For the remainder, we first treat the case $\dim\Null=1$, and then extend to more general representations of cohomogeneity 1.  With $y\in\R$ and $g(\lambda,0)=0$ and $D_\Null g(\lambda,0)=0$ we can write $g(\lambda,y) = y^2g_1(\lambda,y)$. It follows from (ii) that $\sigma(\lambda)=2g_1(\lambda,0)$.  Therefore, by Taylor's theorem,
$$g(\lambda,y) = {\textstyle\frac12}y^2\sigma(\lambda)+ y^3b(\lambda,y)$$
for some smooth function $b$.  Differentiating with respect to $y$ we have
$$g'(\lambda,y) = y\sigma(\lambda) + y^2c(\lambda,y),$$
for some new function $c$.
Now use the fact that $\sigma$ crosses $0$: let $\lambda_1,\lambda_2\in\Lambda$ be such that $\sigma(\lambda_1)>0$ and $\sigma(\lambda_2)<0$.  There is then a neighbourhood $V$ of 0 in $\Null$ such that $\sigma(\lambda_1) + yc(\lambda_1,y)>0$ and $\sigma(\lambda_2) + yc(\lambda_2,y)<0$ for all $y\in V$. Now fix $y\in V$. By the intermediate value theorem there is a $\lambda\in\Lambda$ (by taking any path in $\Lambda$ joining $\lambda_1$ to $\lambda_2$) with $g'(\lambda,y)=0$ as required.

Now consider the case where $\dim\Null>1$, and let $z(y)$ be the generator of the ring of invariants for the $N_K(L)$-action on $\Null$. Then we can write $g=g(\lambda,z)$, and the argument above can be repeated using $z$ in place of $y$.

For the stabilizer, let $m=(\lambda,y,s)$ be a critical point.  Then $s= s(\lambda,y,0)$ (i.e., $\bar s=0$) so the result follows.
\end{proof}

\section{Stability}\label{sec stability}

In this paper we use equations
\eqref{genbundle1}, \eqref{genbundle2}, \eqref{genbundle3} together
with \eqref{re1} and \eqref{re2} to obtain results on the
nonlinear stability, persistence and bifurcations of Hamiltonian
relative equilibria.  To do so, we take advantage of the available freedom
in the isotropy, given by the indeterminacy in the Lie algebra
element $\eta\in \g_z$ in the bundle equations with isotropy, and clarify the appearance in the literature of ``orthogonal velocities'' and optimal stability conditions.  In
this section we will be concerned only about the stability problem of
relative equilibria.  We start by introducing some notions needed in the
following.  We will always assume that we are given a
$G$-Hamiltonian system $(\PP,\w,G,\J,h)$, a point $z\in\PP$ with
split momentum $\J(z)=\mu$, and that we have chosen $G_z$-invariant splittings
$\g_\mu=\g_z\oplus \m$ and $\ker T_z\J=\g_\mu\cdot z\oplus N$, and a
$G_\mu$-invariant splitting $\g=\g_\mu\oplus\q$.  The results in this
section extend the main result of \cite{MoRo11} to allow non-compact $G_\mu$, although we still require $\g$ to admit a $G_\mu$-invariant inner product.

\begin{defn}\label{ortvelocity}
Let $z\in\PP$ be a relative equilibrium and $\xi\in\g_\mu$ a velocity. Then the projection
$\xi^\perp=\Proj_\m(\xi)\in\m$ is called the orthogonal velocity of
$z$.
\end{defn}

Note that the orthogonal velocity of a relative equilibrium is unique once a splitting \eqref{gmusplitting} has been chosen, since all velocities of $z$ differ by elements of $\g_z$. On the other hand, in general there is no canonical orthogonal velocity for $z$, unless the splitting \eqref{gmusplitting} is unique (see Remark \ref{remarkstability} for a counterexample).

\begin{defn}\label{nonlinearstability}
Let $z\in\PP$ be a relative equilibrium with momentum $\J(z)=\mu$.
We say $z$ is nonlinearly stable (or stable modulo $G_\mu$) if for every $G_\mu$-invariant neighbourhood $U$ of $z$ there is an open neighbourhood $O$ of $z$ contained in $U$ such that the integral curve of $X_h$ through any point in $O$ is contained in $U$ for all time.
\end{defn}

This definition of stability for symmetric Hamiltonian systems  was
introduced in \cite{Pat92}  where it is shown to be a natural
one in this setting, since it has been noticed that the existence of
a conserved momentum map can allow a drift of the Hamiltonian
dynamics along the $G_\mu$-orbits which makes the obvious choice of
orbital stability too restrictive in the symmetric Hamiltonian
scenario. We will see below that stability modulo $G_\mu$ is related to the following notion of
formal stability.

\begin{defn}\label{def formalstability}
Let $z\in\PP$ be a relative equilibrium with momentum $\J(z)=\mu$. We
say that $z$ is formally stable if it admits a velocity
$\xi\in\g_\mu$ such that $\ed^2_z h_\xi\rrestr{N}$ is
definite.
\end{defn}

The  lemma below will be of great importance in the exchange of information between a $G$-Hamiltonian system $(\PP,\w,G,\J,h)$ and its local model $(Y,\w_Y,G,\J_Y,h\circ\varphi)$. In particular, it shows that the concept of formal stability is independent of the choice of symplectic normal space $N$.  First we fix some notation.

\begin{defn}\label{branchdef}
We say that an injective map
$\bar z:W\rightarrow\PP$, where $W$ is a neighbourhood of $0$ in a
vector space, is a \emph{parametrized branch of relative equilibria} if for
every $w\in W$, the point $\bar z(w)$ is a relative equilibrium, and all
the relative equilibria of the branch are inequivalent in the
sense that they all belong to different $G$-orbits. Let $\bar z(0)=z$
with $\J(z)=\mu$. The points  of the branch $\bar z(w)$ are said to be \emph{of
the same symplectic type} (or symplectic orbit type) if for every $w\in W$,
there is some $g\in G$ such that $G_{\J(\bar z(w))}=gG_\mu g^{-1}$ and
$G_{\bar z(w)}=gG_zg^{-1}$.
\end{defn}

Notice that by construction of the Marle-Guillemin-Sternberg model, if $\bar z(w)$ is a family of constant symplectic type then the symplectic normal spaces $N_{g^{-1}\bar z(w)g}$ and $N_z$ are isomorphic
as $G_z$-modules.  Since group orbits of relative equilibria consist of relative equilibria, it is customary to refer, by extension, to $G\cdot \bar z(w)$ as a parametrized branch of relative equilibria, although in this case we are not parametrizing points but $G$-orbits of relative equilibria.

For the following result, recall that the function $\bh$ is defined in \eqref{eq:h-bar}.

\begin{lemma}\label{techlemma}
\
\begin{itemize}
\item[(i)]
Let $z=[e,\rho,0]\in Y$, with $\rho\in\m^*$. Then
for every $(\rho,\eta')\in\m^*\times\g_z$,
$$D_N\bh_{\eta'}(\rho,\cdot)=0\quad \text{if and only if}\quad
\ed_{[e,\rho,0]}(h\circ\varphi)_{\xi}=0,$$ where
$$(h\circ\varphi)_{\xi}:=(h\circ\varphi)-\J_Y^\xi(\cdot)
=h_{\xi}\circ\varphi,$$ and
${\xi}=D_{\m^*}\bh(\rho,0)+\eta'$. Furthermore, if
$0\in N$ is a critical point of $\bh_{\eta'}(\rho',\cdot)$,
then
$$D^2_{N}\bh_{\eta'}(\rho,0)=\ed^2_{[e,\rho,0]}(h\circ\varphi)_{\xi}\rrestr{N},$$ where $D^2_{N}\bh_{\eta'}(\rho,0)$ is
the Hessian of the function $\bh_{\eta'}(\rho,\cdot)\in
C^\infty(N)$ at $0\in N$ and
$\ed^2_{[e,\rho,0]}(h\circ\varphi)_{\xi}\rrestr{N}$ is
the restriction to $N$ of the Hessian of
$(h\circ\varphi)_{\xi}$ at $[e,\rho,0]$.

\item[(ii)] Let $z\in\PP$ be a critical point of $h_\xi$ for some $\xi\in\g$. Suppose that $\ed^2_z h_\xi\rrestr{N}$
is non-degenerate, degenerate, positive definite or negative definite.
Then the corresponding property is satisfied
for any other choice of symplectic normal space $N'$.
\item[(iii)] The same conclusions hold if in (ii) we replace $N$ by $N^{K}$ and $N'$ by $(N')^K$ respectively,
for any compact subgroup $K\leq G_z$.
\end{itemize}
\end{lemma}
\begin{proof}

 We will assume, without loss of
generality, that $G_{\J_Y([e,\rho,0]))}=G_\mu$ for each point of the family. In particular, from \eqref{JY}, this assumption implies that
$(G_\mu)_{\rho}=G_\mu$. Recall that from (v) in Proposition
\ref{charRE} the critical points of
$(h\circ\varphi)_{\xi}$ correspond precisely to relative
equilibria of the $G$-Hamiltonian system $(Y,\w_Y,G,\J_Y,h\circ\varphi)$. Now using \eqref{JY} and
$$(h\circ\varphi)([g,\rho,v])=\bh(\rho,v)$$ we have that

\begin{eqnarray*}
\langle\ed(h\circ\varphi)_{\xi}([g,\rho,v]),\gamma_{\lambda,\dot\rho,\dot
v}([g,\rho,v])\rangle & = & \frac{d}{dt}\restr{t=0}(\bh(\rho+t\dot\rho,v+t\dot v)\\
& &
-\frac{d}{dt}\restr{t=0}\langle\Ad^*_{(ge^{t\lambda})^{-1}}(\mu+\rho+t\dot\rho+\J_N(v+t\dot
v)),\xi\rangle\\
& = & D_{\m^*}\bh(\rho,v)\cdot\dot\rho+D_N\bh(\rho,v)\cdot\dot v\\
& & + \langle\Ad^*_{g^{-1}}(\ad^*_\lambda(\mu+\rho+\J_N(v))),\xi\rangle\\
& & -\langle\Ad_{g^{-1}}^*(\dot\rho+D_N\J_N(v)\cdot\dot
v),\xi\rangle
\end{eqnarray*}
It follows that if we make $g=e$ and $v=0$ in the
previous expression we obtain
$\ed(h\circ\varphi)_{\xi}([e,\rho,0])=0$ if and only if
\begin{eqnarray*}
D_{\m^*}\bh(\rho,0)-\Proj_{\m}(\xi) & = & 0\\
D_N\left(\bh-\langle
\J_N,\Proj_{\g_z}(\xi)\rangle\right)(\rho,0) & = & 0\\
\ad_{\xi}^*(\mu+\rho) & = & 0
\end{eqnarray*}
Therefore since $\Proj_{\g_z}(\xi)=\eta'$ the first
condition is automatically satisfied and the second is equivalent to
$$
D_N\bh_{\eta'}(\rho,0)  = 0.$$ Notice that since
$\xi\in\g_\mu$ and $(G_\mu)_{\rho}=G_\mu$ the third
equation is automatically satisfied. This proves the first part of
$(i)$.

For the second part of $(i)$, notice  now that if $0\in N$ is a
critical point of $\bh_{\eta'}(\rho,\cdot)$, then
$D_N\bh(\rho,0)=0$ and $D\J_N^{\eta'}(0)=0$ for any
$\rho\in\m^*$. This follows from \eqref{JN} since
$\J_N^{\eta'}$ is a quadratic polynomial on $N$. We can
naturally identify $N$ with the subspace of $T_{[e,\rho,0]}Y$ given by
$\{\gamma_{0,0,\dot v}([e,\rho,0]) \, : \, \dot v\in N\}$. Then, we
have
$$\begin{array}{l}\ed^2_{[e,\rho,0]}(h\circ\varphi)_{\xi} (\gamma_{0,0,\dot v_1}([e,\rho,0]),\gamma_{0,0,\dot v_2}([e,\rho,0]))\\
\qquad=  D^2_N\bh(\rho,0)(\dot v_1,\dot v_2)-D^2\J_N^{\eta'}(0)(\dot v_1,\dot v_2)\\
\qquad =  D^2_N\bh_{\eta'}(\rho,0)(\dot v_1,\dot v_2)\end{array}$$
with ${\xi}=D_{\m^*}\bh (\rho,0)+\eta'$.

\noindent For (ii), let $N$ and $N'$ be two symplectic normal spaces at $z$. That is, both are $G_z$-invariant complements to $\g_\mu\cdot z$ in $\ker T_z\J$. Every element in $N'$ can therefore be written as $v+\lambda\cdot z$ for some $v\in N$ and $\lambda\in\g_\mu$.   Since the Hamiltonian vector field is equivariant, the point $g\cdot z$ is a relative equilibrium with group velocity $\Ad_g\xi$, and hence for all $g\in G$, the corresponding differential vanishes: $\ed h_{\Ad_g\xi}(g\cdot z) = 0$.
Write $g=\exp(t\eta)$ for $\eta\in\g_\mu$, and differentiate with respect to $t$ at $t=0$ to obtain
$$\ed^2h_\xi(\eta\cdot z,-) -\ed\J_{[\eta,\xi]} = 0,$$
where the differentials are taken at $z$. It follows that for any $v\in \ker\ed\J(z)$ and any $\eta\in\g$, we have
\[\ed^2h_\xi(\eta\cdot z,\,v)=0.\]
It then follows that, given any $\lambda\in\g_\mu$ (so that $\lambda\cdot z\in\ker\ed\J(p)$)
\[\ed^2h_\xi(v+\lambda\cdot z,\,v+\lambda\cdot z)=\ed^2h_\xi(v,\,v),\]
as required.
Part (iii) is a straightforward consequence of the pull-back property of Hessians.
\end{proof}

We can now extend a standard stability result, originally stated in
\cite{LeSi98} and \cite{OrRa99} for the case of relative equilibria
with continuous isotropy, extending the work of \cite{Pat92} on the
nonlinear stability of regular relative equilibria. This also extends the theorem in \cite{MoRo11}, where it was assumed that $G_\mu$ is compact.

\begin{thm}\label{stability}
Let $(\PP,\w,G,\J,h)$ be a $G$-Hamiltonian system and $z\in \PP$ a
relative equilibrium with momentum $\J(z)=\mu$. Suppose that $\g$
admits a $G_\mu$-invariant inner product with respect to the adjoint
representation. Suppose in addition that $z$ admits a velocity
$\xi\in\g_\mu$ such that $\ed^2_z h_\xi\rrestr{N}$ is definite,
where $N$ is any $G_z$-invariant complement to $\g_\mu\cdot z$ in
$\ker T_z\J$. Then $z$ is nonlinearly stable, in the sense of
Definition \ref{nonlinearstability}.
\end{thm}
\begin{proof}
Since $\varphi:Y=\mgs\rightarrow \PP$ is a local $G$-equivariant
symplectomorphism, it is clear that $[e,0,0]$ is a relative
equilibrium with momentum $\J_Y([e,0,0])=\mu$ for the local
$G$-Hamiltonian system $(Y,\w_Y,G,\J_Y,h\circ\varphi)$ that models
the original system in a neighbourhood of the group orbit $G\cdot z$.
It is easy to verify that $N$, identified with $\{\gamma_{0,0,\dot
v}([e,0,0]) \, : \, \dot v\in N\}\in T_{[e,0,0]}(\mgs)$ is a
symplectic normal space at $[e,0,0]$. Let $\xi\in\g_\mu$ be a
velocity for $[e,0,0]$ and let $\xi=\xi^\perp+\eta\in\m\oplus\g_z$ according to the splitting \eqref{gmusplitting}.
Then, it follows from (i) in Lemma \ref{techlemma} that
$D_N\bh_\eta(0,0)=0$ and $D_{\m*}\bh (0,0)=\xi^\perp$. Suppose now
that $\ed^2_{[e,0,0]}(h\circ\varphi)_\xi\rrestr{N}$ is definite.
Then, from (ii) also in Lemma \ref{techlemma} we have that
$D^2_N\bh_\eta (0,0)$ is definite. The existence of a
$G_\mu$-invariant inner product on $\g$ implies that $\mu$ is split,
so (\ref{genbundle1}--\ref{genbundle3}) apply.

By the smoothness of the dependence of $\bh$ on $\m^*$, it follows
from the Morse lemma that $0\in N$ is a Lyapunov stable point for
\eqref{genbundle3}. Also, the condition on the existence of the
inner product guarantee that the coadjoint orbits of $G_\mu$ are
contained in compact hypersurfaces, so they are bounded. This, using
\eqref{genbundle2} implies that $0\in \m^*$ is Lyapunov stable for
$\eqref{genbundle2}$.  Since $D_{\m^*}h(\rho,v)+\eta\in\g_\mu$
for all $(\rho,v)\in \m^*\times N$, it follows from
$\eqref{genbundle1}$ that $g(t)\in  g_0G_\mu$ for any initial
condition $g(0)=g_0$ and all $t$.

  Let
$O_G\subset G$ be an open neighbourhood of $e$. Any $G_\mu$-invariant
neighbourhood of $[e,0,0]$ in $\mgs$ is given by $U=\{G_\mu\cdot
[g,\rho,v]\,:\,g\in U_G,\,\rho\in U_{\m^*},\,v\in U_N\}\subset \mgs$
where $U_G\subset G,\,U_{\m^*}\subset \m^*$ and $U_N\subset N$ are
neighbourhoods of $e,0$ and $0$ respectively. It follows from the
above discussion, and from the Lyapunov stability of $(0,0)$ in
$\m^*\times N$ that we can find open neighbourhoods $O_{\m^*}\subset
U_{\m^*}$ and $O_N\subset U_N$ of the origins such that $\rho(t)\in
U_{\m^*}$  and $v(t)\in U_N$ for all $t$ if $\rho(0)\in O_{\m^*}$
and if $v(0)\in O_{N}$. Therefore, calling $O=\{[g,\rho,v]\,:\,g\in
U_G,\,\rho\in O_{\m^*},\,v\in O_N\}\subset \mgs$, we have that the
integral curves of $X_{h\circ\varphi}$ through points in $O$ always
lie inside $U$, proving the nonlinear  stability of $[e,0,0]$, hence of $z$.
It follows from Lemma \ref{techlemma}(ii) that this is independent of the
choice of symplectic normal space $N$.
\end{proof}

\begin{rem} \label{remarkstability}
Results analogous to Theorem \ref{stability} have been obtained first in \cite{Pat92} for the case when $G_z$ is discrete and in \cite{LeSi98,OrRa99} under similar hypotheses as here. In \cite{LeSi98} the result is proved using the symplectic slicing technique, and the proof of  \cite{OrRa99} is based on an extension of the methods of \cite{Pat92} to the case of continuous isotropy.  Theorem \ref{stability} is, however, more general, since it could guarantee nonlinear stability in some cases in which the application of \cite{LeSi98,OrRa99} is inconclusive. The main difference with our result is that in those two references, the bilinear form on $N$ used to test stability is $\ed_zh_{\xi^\perp}$, where $\xi^\perp$ is the orthogonal velocity associated to some choice of $G_z$-invariant splitting \eqref{gmusplitting}. The strongest stability results follow then by testing over all possible orthogonal velocities corresponding to all possible such invariant splittings. In our approach, we fix from the beginning one splitting, which is necessary only for the construction of the local model of the $G$-Hamiltonian system, and we test over all possible velocities admissible for the relative equilibrium under study. This is the same as testing stability for $\xi^\perp+\eta$ for all possible $\eta\in\g_z$, where $\xi^\perp$ is the orthogonal velocity corresponding to the fixed invariant splitting of $\g_\mu$. We provide now an example where the methods of \cite{LeSi98,OrRa99} don't predict stability, but Theorem \ref{stability} does.  This is a different example to the one in \cite{MoRo11}, and here the relative equilibrium is not an equilibrium.

Let $z$ be a relative equilibrium with $G_\mu=\mathrm{SO}(3)$ and $G_z=S^1$. If we identify $\g_\mu$ with $\R^3$ and $\g_z$ with $\operatorname{span}\langle\bf{e_3}\rangle$, the only $G_z$-invariant complement to $\g_z$ is given by $\m=\operatorname{span}\langle\bf{e_1},\bf{e_2}\rangle$, which can be identified with $\m^*$ using the standard inner product on $\R^3$.

Suppose that the symplectic normal space $N$ at $z$ is isomorphic to $\R^4$, where in a Darboux basis the symplectic matrix $\Omega$ takes the form
$$\Omega=\left(\begin{array}{llll}
0 & 1 & 0 & 0 \\
-1 & 0 & 0 & 0\\
0 & 0 & 0 & 1\\
0 & 0 & -1 & 0\end{array}\right).$$
Suppose that in the linear coordinates $\{x_1,y_1,x_2,y_2\}$ associated to this basis we have
$$\bh(\rho,v)=\frac 12 \left(x_1^2+y_1^2-x_2^2-y_2^2\right)+f(\rho),$$
where $v=(x_1,y_1,x_2,y_2)$, $\rho\in\m^*$ and $f$ is an arbitrary $S^1$-invariant function on $\m^*$.
Suppose also that in this basis for $N$, the symplectic action of $G_z=S^1$ is given by
$$S^1\ni\theta\mapsto\left(\begin{array}{ll} R_\theta & 0 \\ 0 & R_\theta\end{array}\right),$$
with $R_\theta$ the standard rotation through $\theta$ on $\R^2$. The associated momentum map has the expression
$$\J_N(v)=\frac{1}{2} \left(x_1^2+y_1^2+x_2^2+y_2^2\right).$$
It follows from Lemma \ref{techlemma} that testing definiteness of $\ed_zh_{\xi^\perp}$, with $\xi^\perp$ being the orthogonal velocity relative to the invariant splitting of $\R^3$ into $\g_z$ and $\m$ is equivalent to testing definiteness of
$$D_N^2\bh_0(0,0)=\operatorname{diag}(1,1,-1,-1),$$
and the test is inconclusive. Since there are no more invariant splittings available, the results of \cite{LeSi98,OrRa99} can't predict the nonlinear stability of this relative equilibrium. In our approach, using Theorem \ref{stability} it follows from Lemma \ref{techlemma} that the relative equilibrium will be nonlinearly stable if there is some $\eta\in\R\simeq\g_z$ for which $D_N^2\bh_\eta(0,0)$ is definite. We have the general expression
$$\bh_\eta(\rho,v)=\frac{1-\eta}{2}\left(x_1^2+y_1^2\right)-\frac{1+\eta}{2}\left(x_2^2+y_2^2\right)+f(\rho).$$
Then we obtain
$$D_N^2\bh_\eta(0,0)=\operatorname{diag}(1-\eta,1-\eta,-(1+\eta),-(1+\eta)),$$
which is definite for $\eta\in (-\infty,-1)\cup(1,+\infty)$, and therefore the relative equilibrium is nonlinearly stable.

Finally, we remark that the first  time that an approach similar to
ours was used in the context of testing nonlinear stability of
Hamiltonian relative equilibria is in \cite{Wu01}. In that
reference, the authors only consider free Hamiltonian actions, which
do not exhibit isotropy, however in that situation they investigate
more general momentum values without requiring the existence of a
$G_\mu$-invariant inner product on $\g$.
\end{rem}

\section{Persistence}\label{sec persistence}

The problem of persistence of relative equilibria consists in
providing conditions under which a given relative equilibrium $z$
belongs to a continuous family of relative equilibria, called the
persisting branch. Additional properties about this family, such as the
stability and stabilizers of its elements, or its geometric
features, are  usually also of interest.
\begin{defn}\label{persistence}
 A relative equilibrium $z$ is said to persist, if for every $G$-invariant neighbourhood $U$ containing $z$, the set $U\backslash G\cdot z$ contains a relative equilibrium.
\end{defn}
In several references, it is also required in the definition of
persistence that the persisting set of relative equilibria is the
$G$-saturation of a parametrized branch of relative equilibria of
the same symplectic type, as in Definition \ref{branchdef}. We will not require
this. In this section we will examine four different scenarios where
a given relative equilibrium can persist, together with additional
information about the persisting branch.

\subsection{Persistence in the case of non-degeneracies}\label{subsec persistence nondeg}
\ \\
The following theorem gives sufficient conditions for finding
persisting branches of relative equilibria starting with a relative
equilibrium satisfying a typical non-degeneracy hypothesis, as well
as an estimate on the orbit types of the persisting relative
equilibria. It extends Theorem 3.2 of \cite{RoSD97} to non-compact groups,
and uses a weaker hypothesis.

First recall that the orbit type of a point $z\in M$ is the conjugacy class in $G$ of its stabilizer $G_z$ and it is denoted by $(G_z)$. In particular, two distinct relative equilibria have the same orbit type if and only if they have conjugate stabilizers.

If $M$ is a smooth manifold on which the Lie group $G$ acts smoothly and properly, for each orbit type $(H)$, the set
$$\PP_{(H)}=\{z\in M\, : \, G_z \,\text{is conjugate to}\, H\},$$
called the manifold of orbit type $(H)$ in $\PP$, is a disjoint
union of connected embedded submanifolds of $\PP$.  Our main
persistence result in the case of non-degeneracies is the following.

Before stating the theorems, we discuss briefly an algebraic hypothesis we make below and in other theorems:

\begin{lemma}\label{lemma:cocentral}
Let $H$ be a compact subgroup of the Lie group $G$, with Lie algebras $\h<\g$ respectively.  Suppose one can write $\g=\h\oplus\m$ (as vector spaces), with $\m\subset\z(\g)$, the centre of $\g$.  Then
\begin{itemize}
\item[(i)] the decomposition is $\ad(\h)$-invariant;
\item[(ii)] The decomposition $\g=\h\oplus\m$ is an equality of Lie algebras, with $\m$ Abelian.
\item[(iii)] $[\g,\g]\subset\h$;
\item[(iv)] any $\ad(\h)$-invariant decomposition (in particular if it is $\Ad(H)$-invariant) $\g=\h\oplus\m'$ also satisfies $\m'\subset\z(\g)$;
\end{itemize}
Furthermore, if $G$ is compact then $\m\subset\z(\g)$ is equivalent to $[\g,\g]\subset\h$.
\end{lemma}

When this condition that $\m\subset\z(\g)$ is satisfied, we say $\h$ is a \emph{co-central} subalgebra of $\g$.  It holds for any subalgebra if $\g$ is Abelian, and a non-Abelian example is provided by $H=\mathsf{SU}(n)$ as a subgroup of $G=\mathsf{U}(n)$.

\begin{proof}
(i) For the decomposition to be $\ad(\h)$-invariant simply means $[\h,\m]\subset\m$, but since $\m\subset\z(\g)$, we have $[\h,\m]=0\subset\m$.

\noindent(ii) This is clear, since $[\h,\m]=[\m,\m]=0$.

\noindent(iii) This follows immediately from (ii).

\noindent(iv) Since $\g=\h\oplus\m=\h\oplus\m'$, the projection $\g\to\m$ with kernel $\h$ induces an isomorphism $f:\m'\to\m$.  Moreover, since $\m'$ is $\ad(\h)$-invariant, the projection commutes with $\ad(\h)$, so the two representations $\m$ and $\m'$ are equivalent.  Thus $[\h,\m']=0$. Moreover, since $\m\subset\z(\g)$ it follows that $[\m,\m']=0$. Therefore $[\g,\m']=0$, whence $\m'\subset \z(\g)$, as required.

For the final part, suppose $G$ is compact,
and $[\g,\g]\subset\h$.   It follows that $\h\lhd\g$ (an ideal) since $[\g,\h]\subset[\g,\g]\subset\h$, and hence (by compactness of $G$), there is a decomposition $\g=\h\oplus\m$ that is $\ad(\g)$-invariant.   In particular, this implies $[\g,\m]\subset\m$. But since $[\g,\m]\subset\h$ by hypothesis, we conclude $[\g,\m]=0$, whence $\m\subset\z(\g)$.
\end{proof}

Note that since the adjoint action is by Lie algebra isomorphisms, fixed point sets in $\g$ are necessarily subalgebras (see Lemma \ref{lemma:algebraic}).  In particular if $K\leq G_z$ then $(\g_\mu)^K$ is a subalgebra of $\g_\mu$, and moreover (since $K$ is then compact), $(\g_\mu^*)^K$ is naturally the dual of $\g_\mu^K$. The lemma above will be applied with $\g_z^K$ in place of $\h$ and $\g_\mu^K$ in place of $\g$.

\begin{thm}\label{thm persistence nondeg} Let $z\in\PP$ be a relative equilibrium with momentum  $\J(z)=\mu$, and let $\xi\in\g_\mu$ be a velocity for $z$.
Write $\xi=\xi^\perp + \eta\in\m\oplus\g_z$ according to the splitting \eqref{gmusplitting}.
Suppose there exists a subgroup $K\leq (G_z)_\eta$ such that
\begin{itemize}
\item[(i)] $\g_z^K$ is a co-central subalgebra of $\g_\mu^K$, and
\item[(ii)] $\ed^2_zh_{\xi}\rrestr{N^K}$ is non-degenerate.
\end{itemize}
Then, there is a smooth $N_{(G_z)_\eta}(K)$-equivariant map
\begin{eqnarray*}\bar z:{\m^*}^K\times {\g_z}^K & \longrightarrow & \PP^K\\
(\rho,\eta') & \longmapsto & \bar z(\rho,\eta')
\end{eqnarray*}
defined in a neighbourhood of $(0,0)$ such that $\bar z(0,0)=z$, and  for
each $\eta'$ the map $\rho\mapsto \bar z(\rho,\eta')$ is an immersion, such that for each
$(\rho,\eta')$ in the domain, the point $\bar z(\rho,\eta')$ is a relative
equilibrium with velocity of the form $\xi+\eta'\in\g_\mu$ and stabilizer
$G_{\bar z(\rho,\eta')}$ satisfying
$$K\leq G_{\bar z(\rho,\eta')}\leq (G_z)_\rho.$$
Moreover the branch $\bar z(\rho,\eta')$ consists of every relative
equilibrium near $z$ with stabilizer containing $K$  and velocity
$\xi'$ satisfying $\xi'-\xi\in\g_z^K$.
\end{thm}

This theorem if only non-trivial if ${\m^*}^K\neq0$. This is because, if ${\m^*}^K=0$, then the theorem is satisfied by the constant map $\bar z(\eta')=z$, since $z$ is by assumption a relative equilibrium with velocity $\xi+\eta'$ for every $\eta'\in\g_z$.

\begin{proof}
Write $\xi=\xi^\perp + \eta\in\m\oplus\g_z$ according to the splitting \eqref{gmusplitting}, thus  $\eta:=\Proj_{\g_z}(\xi)\in\g_z$. Using the local model
$(Y,\w_Y,G,\J_Y,h\circ\varphi)$ around the group orbit $G\cdot z$,
and putting $f(\eta',\rho,v)=\bh_{\eta+\eta'}(\rho,v)$ the hypotheses
imply that $D_Nf(0,0,0)=0$ and $D^2_Nf(0,0,0)\rrestr{N^K}$ is
non-degenerate (see \eqref{eq:h-bar} for definition of $\bh$). Therefore, from Lemma \ref{lemmaND} it follows that
there exists a local smooth $N_{(G_z)_\eta}(K)$-equivariant map
$v:{\m^*}^K\times \g_z^K\rightarrow N^K$ satisfying
$D_Nf(\eta',\rho,v(\rho,\eta'))=D_N\bh_{\eta+\eta'}(\rho,v(\rho,\eta'))=0$.
Therefore \eqref{req2} is satisfied for the isotropy Lie algebra element $\eta+\eta'$.
In the local model, we therefore put $\bar z(\rho,\eta') = (\rho,v(\rho,\eta'))$, which for each $\eta'$ is clearly an immersion.
We
will now show that \eqref{req1} is also satisfied for the same
choice of element.
Using the $G_z$-invariant (and therefore $K$-invariant) splittings $\g_\mu=\m\oplus\g_z$ and
$\g_\mu^*=\m^*\oplus\g_z^*$, and noting that
$\rho+\J_N(v(\rho,\eta'))\in\m^*\oplus\g_z^*=\g_\mu^*$ we can
identify $\rho+\J_N(v(\rho,\eta'))$ with an element $\lambda\in
\g_\mu^*$.
Moreover, $\lambda\in{\g_\mu^*}^K$, since $\rho\in {\m^*}^K$
and by the equivariance of $\J_N$ and the fact that
$v(\rho,\eta')\in N^K$, $\J_N(v(\rho,\eta'))\in{\g_z^*}^K$.
Therefore, \eqref{req1} is equivalent to
\begin{equation}\label{eq:projection to m^*}
\Proj_{\m^*}\left(\ad^*_\gamma\lambda\right)=0,
\end{equation}
with $\gamma=D_{\m^*}\bh(\rho,v(\rho,\eta'))+\eta+\eta'$.  We claim that $\gamma\in\g_\mu^K$. It then follows from hypothesis (i) that the equation above is satisfied. To see  this, note that \eqref{eq:projection to m^*} is equivalent to
$$\langle\ad^*_\gamma\lambda,\,\beta\rangle = 0,\quad \forall\beta\in\m^K.$$
This in turn is equivalent to $\langle\lambda,\,[\gamma,\beta]\rangle = 0$ (for all $\beta\in\m^K$).
The hypothesis implies $\m^K\subset\z(\g_\mu^K)$, whence $[\beta,\gamma]=0$ for all $\beta\in\m^K,\;\gamma\in\g_\mu^K$ showing \eqref{eq:projection to m^*} to be satisfied, independently of the values of $\gamma\in\g_\mu^K$ and $\lambda\in{\g_\mu^*}^K$.

To prove the claim that $\gamma\in\g_\mu^K$, recall that $K\leq (G_z)_\eta$, and therefore $\eta+\eta'\in \g_z^K$. It remains to show that $D_{\m^*}\bh(\rho,v(\rho,\eta'))\in \m^K$.
Now $\bh$ is $G_z$-invariant, and therefore for each $v\in N^K$, the function $\rho\mapsto\bh(\rho,v)$ is $K$-invariant. It follows (from the Principle of Symmetric Criticality) that for $\rho\in{\m^*}^K$ we have $D_{\m^*}\bh(\rho,v)\in\m^K$ as required.

Since we see that \eqref{req1} and \eqref{req2} are simultaneously
satisfied, if we define $\bar z(\rho,\eta')\in\PP$ by
$\bar z(\rho,\eta')=\varphi([e,\rho,v(\rho,\eta')])$ then it follows
by the equivariance of $\varphi$, Proposition \ref{prop cond RE} and
the fact that $\varphi$ is a local isomorphism of symmetric
Hamiltonian systems that for each pair
$(\rho,\eta')\in{\m^*}^K\times \g_z^K$, the point $\bar z(\rho,\eta')\in
\PP$ is a relative equilibrium  with velocity $\xi+\eta'$.

The stabilizer of $\bar z(\rho,\eta')$ is $G_{\bar z(\rho,\eta')}$. Since
$\varphi$ is $G$-equivariant this is equal to
$G_{[e,\rho,v(\rho,\eta')]}=(G_z)_m$ where $(G_z)_m$ is the
stabilizer of $m=(\rho,v(\rho,\eta'))\in \m^*\times N$ with
respect to the diagonal $G_z$-action. According to Lemma
\ref{lemmaND},  $K\leq (G_z)_m$ (so $z\in \PP^K$) and is contained in
$(G_z)_\rho$. It also follows from Lemma \ref{lemmaND} that the
pairs $(\rho,v(\rho,\eta'))\in{\m^*}^K\times\g_z^{K}$ give all
the solutions to equation \eqref{re2} with fixed $\xi$ and $\eta$,
from which the last statement follows.
\end{proof}

\subsection{Persistence of branches of constant orbit type}\label{subsec persistence maxim}
\ \\
As an application of Theorem \ref{thm persistence nondeg}, we now study in
more detail under which conditions we can guarantee that a relative equilibrium
belongs to a smooth branch of relative equilibria of the same orbit type (and possibly the same symplectic type), parametrized by momentum values. The following is an extension within our framework of the results of Patrick \cite{Pat95} in the free case and Lerman and Singer \cite{LeSi98} in the singular case, each reference requiring slightly different hypotheses.
It is also related to results of Patrick and Roberts \cite{PaRo00}, where $G$ is assumed to be compact and acting freely, and they give, \emph{inter alia}, necessary and sufficient conditions for the set of relative equilibria to form a symplectic submanifold.  It would be interesting to adapt their approach to non-free actions.

\begin{thm}\label{thm persistence same orbit type}
Let $z\in \PP$ be a relative equilibrium with momentum $\mu=\J(z)$ for which
$\ed^2_zh_\xi\rrestr{N^{G_z}}$ is non-degenerate for some admissible velocity $\xi\in \g_\mu.$
Suppose  moreover that the Lie algebra $\lil$ of $L:=N_{G_\mu}(G_z)/G_z$ is Abelian, where $N_{G_\mu}(G_z)$ is the normalizer of $G_z$ in $G_\mu$.
\begin{itemize}
\item[(i)]  There is a $G$-invariant neighbourhood $U$ of $z$ in $\PP_{(G_z)}$ such that the set $\Z$ of relative equilibria in $U$ forms a smooth submanifold of dimension $\dim G/G_z+\dim(\m^*)^{G_z}$.  For each $\mu'\in\J(U)$ there is a relative equilibrium $z'\in\Z$ with $\J(z')=\mu'$.

\item[(ii)] Suppose in addition that $G_z\lhd G_\mu$, or more generally  $\LieAlg(N_{G_\mu}(G_z))=\g_\mu$.  Then the submanifold $\Z$ is a symplectic submanifold of $\PP$.

\item[(iii)]  If $G_\mu$ is compact and Abelian
then the momentum values of the relative equilibria of $\Z$ have stabilizers
conjugate to $G_\mu$, and the family $\Z$ is of constant symplectic type.

\item[(iv)] Finally, still assuming the connected component of $G_\mu$ to be a torus, suppose $z$ is formally stable. Then all the relative equilibria of this branch close enough to $z$ are also formally stable and therefore nonlinearly stable.
\end{itemize}
\end{thm}

\begin{proof}
(i)  In the local model $Y=G\times_{G_z}(\m^*\times N)$ of $\PP$ near $G\cdot z$, the point $z$ corresponds to $[e,0,0]$. According to the property
$$G_{[g,\rho,v]}=g((G_z)_\rho\cap (G_z)_v)g^{-1},$$
the points of orbit type $(G_z)$ are those of the form
$$[g,\rho,v],\quad\forall g\in G,\,\rho\in {\m^*}^{G_z},\,v\in N^{G_z}.$$
We wish to find solutions of (\ref{re1}, \ref{re2}) in $Y_{(G_z)}$.  As usual, the velocity $\xi$ can be decomposed as $\xi=\xi^\perp+\eta\in\m\oplus\g_z$, see  \eqref{gmusplitting}. In the hypothesis of the theorem, we have that
$D_N\bh_\eta(0,0)=0$ and $D^2_N\bh_\eta(0,0)\rrestr{N^{G_z}}$ is non-degenerate. Using that $\J_Y$ is a homogeneous quadratic polynomial on $N$ and that $\eta\cdot v=0$ for every $\eta\in \g_z$ and $v\in N^{G_z}$, we have that $D_N\bh(0,0)=D_N\bh_\eta(0,0)=0$ and that $D^2_N\bh(0,0)\rrestr{N^{G_z}}=D^2_N\bh_\eta(0,0)\rrestr{N^{G_z}}$ is non-degenerate.

By the Principle of Symmetric Criticality, and the pull-back
property of Hessians, this means that the function
$\bh^{G_z}(\rho,\cdot):=\bh(\rho,\cdot)\rrestr{N^{G_z}}$, with
$\rho\in{\m^*}^{G_z}$, has a non-degenerate critical point at $0\in
N^{G_z}$ for $\rho=0$. Therefore, from the implicit function theorem
there is a smooth function $v:{\m^*}^{G_z}\rightarrow N^{G_z}$ such
that for any $g\in G$ and $\rho\in {\m^*}^{G_z}$, the points in the
branch $[g,\rho,v(\rho)]$ are all the solutions of \eqref{re2}, with
$\eta=0$, with orbit type $(G_z)$ for $\rho$ sufficiently close to
$0$.

At points of this  branch, condition \eqref{re1} is equivalent to
$\ad^*_{D_{\m^*}\bh(\rho,v(\rho))}\rho=0$, again because $\J_N(v)=0$ for $v\in N^{G_z}$.

To see that this equality is satisfied for all $\rho$ in a neighbourhood of 0, notice that since $\rho\in{\m^*}^{G_z}$ and
$v(\rho)\in N^{G_z}$ then by the Principle of Symmetric Criticality,
$D_{\m^*}\bh(\rho,v(\rho))\in\m^{G_z}$ (as in the proof of Theorem \ref{thm persistence nondeg}). Since $\m^{G_z}$, together with the  bracket induced from $\g$ is
isomorphic to the Lie algebra $\mathfrak{l}$, and
${\m^*}^{G_z}$ is isomorphic to $\left(\m^{G_z}\right)^*=\mathfrak{l}^*$, the hypothesis on the Abelian nature of $\lil$ implies \eqref{re1}.

By \eqref{JY}, and since $\J_N(v)=0$ for $v\in N^{G_z}$, the  values
of $\J\rrestr{\PP_{(G_z)}}$ near $\mu$ form a neighbourhood of $\mu$ in the set
$$\{\Ad^*_{g^{-1}}(\mu+\rho)\,:\, g\in G,\,\rho\in {\m^*}^{G_z}\}.$$
Let $\mu'$ be one of such values, close to the coadjoint orbit through $\mu$. Then, there are $g\in G$ and $\rho\in {\m^*}^{G_z}$ close to 0, such that $\mu'=\Ad^*_{g^{-1}}(\mu+\rho)$. It follows that the point $[g,\rho,v(\rho)]$ is a relative equilibrium of orbit type $(G_z)$ close to the orbit $G\cdot z$.

Furthermore, if $\rho\in(\g_\mu^*)^{G_\mu}\cap\m^*\subset(\m^*)^{G_z}$, then $\mu'$ has orbit type $(G_\mu)$ and the resulting relative equilibrium will have the same symplectic type as $z$.

(ii) The symplectic form on $Y$ is given by $\omega_Y$ in \eqref{wY}. Pulling this back to the submanifold $\Z$ of relative equilibria in (i) gives a closed form.  To show it is non-degenerate in a neighbourhood of $[e,0,0]$ it suffices to show it is non-degenerate at that point. Let $\gamma_{1,2}$ be arbitrary tangent vectors at this point:
$$
\gamma_{1,2}=T\pi (g\cdot\lambda_{1,2},\dot\rho_{1,2},Dv(\rho)\cdot \dot\rho_{1,2})\in T_{[g,\rho,v(\rho)]}Y.
$$
Note that $\dot\rho_{1,2}\in(\m^*)^{G_z}$, but by Lemma \ref{lemma:algebraic}(iii) the hypothesis $G_z\lhd G_\mu$ implies $\m^{G_z}=\m$, and so
$$T_z\Z\simeq\g\cdot z \times(\m^*)^{G_z} = \g\cdot z \times\m^*.$$

It follows from \eqref{wY} that at $[e,0,0]$,
\begin{eqnarray*}
\w_Y(\gamma_1,\gamma_{2}) &=&
\langle\dot\rho_2,\lambda_1\rangle-\langle\dot\rho_1,\lambda_2\rangle+\langle\mu,[\lambda_1,\lambda_2]\rangle \\
&&\quad {} +\Omega_N(Dv(\rho)\cdot\dot\rho_1,Dv(\rho)\cdot\dot\rho_2),
\end{eqnarray*}
since $\J_N\rrestr{N^{G_z}}=T_{v(\rho)}\J_N\rrestr{N^{G_z}}=0$. We can decompose an element $\lambda\in\g$ as
$$\lambda=\lambda^{\g_z}+\lambda^\m+\lambda^\q,$$
with respect to the fixed splitting $\g=\g_z\oplus \m\oplus \q$. Then
\begin{eqnarray*}
\w_Y(\gamma_1,\gamma_{2}) &=&
\langle\dot\rho_2,\lambda_1^\m\rangle-\langle\dot\rho_1,\lambda_2^\m\rangle+\langle\mu,[\lambda_1^q,\lambda_2^q]\rangle \\
&&\quad {} +\Omega_N(Dv(\rho)\cdot\dot\rho_1,Dv(\rho)\cdot\dot\rho_2).
\end{eqnarray*}
The first two terms are the standard symplectic form on $\m\times\m^*$, the last term can be seen as a `magnetic form' on the same space, which does not alter the non-degeneracy, while the third term is the standard KKS symplectic form on $\g_\mu^\circ=\q^*$. It follows that $\omega_Y$ is indeed non-degenerate on $T_z\Z\simeq \g\cdot z \times\m^* = (\m\times\q)\times\m^*$. (See also Lemma 4.2 of \cite{LeSi98}.)

(iii) If $G_\mu$ is compact then, using Palais' tube theorem, one can choose a $G_\mu$-invariant slice to the coadjoint orbit through $\mu$, which can be identified with $\g_\mu^*$.  Then the stabilizer of $\J_Y([e,\rho,v(\rho)])=\mu+\rho$ is  $(G_\mu)_\rho$, and since $G_\mu$ is Abelian, $(G_\mu)_\rho=G_\mu$.  Thus all  points of $\Z$ have the same symplectic type.

(iv) Finally, suppose $z$ is formally stable; that is, $\ed^2_z h_\xi\rrestr{N}$ is
definite for some admissible velocity $\xi\in\g_\mu$.  Since by (iii) the points of $\Z$ are of the same symplectic type, they have isomorphic symplectic normal spaces $N'\simeq N$.  By continuity, the appropriate nearby hessian $\ed^2_{z'} h_{\xi'}\rrestr{N'}$ is also positive definite, and hence by Theorem \ref{stability} the relative equilibrium is nonlinearly stable.
\end{proof}

\subsection{Persistence in the case of degeneracies}\label{subsec persistence deg}
\ \\
We now prove a persistence result specific for relative equilibria
with continuous isotropy, which exploits the indeterminacy of the
velocity in order to predict branches of persisting relative equilibria.

\begin{thm}\label{thm persistence deg}
Let $z$ be a relative equilibrium with momentum $\J(z)=\mu$ and suppose that $z$ admits the velocity $\xi\in\g_\mu$ with $\xi=\xi^\perp + \eta\in\m\oplus\g_z$. Let $K= (G_z)_\eta$ and suppose that there
exists a subgroup $L\leq K$ for which $\Null:=\ker \ed^2_z h_{\xi}\rrestr{N^L}$ is non-zero.  Suppose moreover the following conditions are satisfied
\begin{itemize}
\item[(i)]  $\g_z^L$ is co-central in $\g_\mu^L$,
\item[(ii)] the action of $N_{K}(L)$ on $\Null$ is of cohomogeneity one, and
\item[(iii)] the eigenvalue  $\sigma(\eta')$ of $\ed^2_z h_{\xi+\eta'}\rrestr{\Null}$ for $\eta'\in\g_z^K$, crosses 0 at $\eta'=0$.
\end{itemize}
Then, for each $v\in\Null \backslash \{0\}$ close to $0$ there are
elements $\eta'_v\in \g_z^K$ close to $0$ and $z_v\in\PP$ close to $z$
which are relative equilibria with velocity
$\xi+\eta'_v$ and stabilizer $K_v$, which contains $L$.
\end{thm}

In Section\,\ref{sec example2}, we give an elementary illustration of this theorem in the system of pairs of point vortices on the sphere.

\begin{proof}
The proof is a direct application of Lemma \ref{lemmaD}, where we choose
$W=\g_z^K$,
and $f(\eta',v)=\bh_{\eta+\eta'}(0,v)$ (with $\rho=0$).  Notice that if $z$ corresponds to $[e,0,0]$ in the local model, then Lemma \ref{lemmaD},
predicts that for each
$v\in\Null$ there is an $\eta'_v$ such that $v$
is a critical point of $\bh_{\eta+\eta'_v}(0,\cdot)$. Therefore, $(0,v)$
satisfies \eqref{req2} for $\eta+\eta'_v$. By an argument similar to
the one used in Theorem \ref{thm persistence nondeg}, the Principle of Symmetric Criticality  guarantees that
$D_{\m^*}\bh(0,v)+\eta+\eta'_v\in\g_\mu^L$, and then condition (i) implies
that \eqref{req1} is also satisfied. Therefore, in the local model the persisting relative
equilibria predicted by this theorem are group orbits of points of the form $[e,0,v]$ for each
$v\in\Null$.  These are the points that correspond to $z_v$ in $\PP$.

The stabiliser is given by $(G_z)_v\cap (G_z)_{\eta'}$.  But $\eta'\in\g_z^K$, so $(G_z)_{\eta'}=K$, and $v\in N^L$ whence $L\leq (G_z)_v$.
\end{proof}

\subsection{Persistence from formally stable relative equilibria}\label{subsec persistence formal}
\ \\
To end this section, we prove
a more general result on momentum parametrized branches persisting
from formally stable relative equilibria. It was originally
obtained in \cite{MoTo03} in a slightly different setup, but there it was required that
$G_\mu$ be compact (see the remark below for more details).  Here we
present it as a product of the bundle equations formalism.

\begin{thm}\label{thm formally stable}
Let $z$ be a formally stable relative equilibrium with $\mu=\J(z)$. Suppose that $\g_\mu$ has a $G_\mu$-invariant inner product (for example, $G_\mu$ is compact). Then there is a $G$-invariant neighbourhood $U$ of $z$ such that for every $\mu'\in \J(U)$ near $\mu$ there is a relative equilibrium $z'$ near $z$ with $\J(z')=\mu'$.
\end{thm}

\begin{proof}
Let $U_0$ be a neighbourhood of $z$ for which the Marle-Guillemin-Sternberg normal form is valid. Recall that for a point $[g,\rho,v]\in Y:=G\times_{G_z}(\m^*\times N)$ we have
$$\J([g,\rho,v])=\Ad_{g^{-1}}^*(\mu+\rho+\J_N(v))$$
Therefore $\J(Y)=\{\Ad_{g^{-1}}^*(\mu+\rho+\epsilon)\,:\, g\in G,\,\rho\in\m^*,\,\epsilon\in\im \J_N\}$.  Let $Y_0\subset Y$ be the subset corresponding to $U_0$; then $\J(Y_0)\subset \J(Y)$.  If $G_\mu$ is compact then $\J(Y_0)$ is an open subset of $\J(Y)$, but it is not known in the more general setting here---see also Remark \ref{rem:G-open} below.

\paragraph{Step 1. Taking $\mu'=\mu+\bar\rho+\bar\epsilon$}\

Suppose $z'\in Y_0$ (or $U_0$) is a relative equilibrium with
$\J(z')=\Ad_{g^{-1}}^*(\mu+\bar\rho+\bar\epsilon)$. This is
equivalent to $g^{-1}\cdot z'$ being a relative equilibrium with
momentum
$\Ad_g^*(\Ad_{g^{-1}}^*(\mu+\bar\rho+\bar\epsilon))=\mu+\bar\rho+\bar\epsilon$.
Therefore in the statement of the theorem we can choose, without
loss of generality, $\mu'=\mu+\bar\rho+\bar\epsilon$.

\paragraph{Step 2. Using the bundle equations}\\
According to Proposition\,\ref{prop cond RE}, a point $[g,\rho,v]$ is a RE if and only if
there exists $\eta\in\g_z$ such that
\begin{eqnarray}\label{e1}
\ad^*_{D_{\m^*}\bh(\rho,v)+\eta}(\rho+\J_N(v)) & = & 0\\
\label{e2}D_N\bh_\eta(\rho,v) & = & 0.
\end{eqnarray}
Equation \eqref{e1} is an equation in $\g_\mu^*$ and is equivalent to
\begin{equation}\label{e3}
\left(D_{\m^*}\bh(\rho,v)+\eta\right)\,\rrestr{T_{\rho+\J_N(v)}\OO}=0
\end{equation}
where $\OO=\{\Ad^*_{l^{-1}}(\rho+\J_N(v))\,:\,l\in G_\mu\}\subset
\g^*_\mu$ is the $G_\mu$-orbit of $\rho+\J_N(v)$ for the coadjoint
action.

Now we need to use the formal stability of $z=[e,0,0]$.  By Lemma\,\ref{techlemma}(i), this is equivalent to there being a $\gamma\in\g_z$ (which we fix throughout) such that $D_N^2\bh_\gamma$ is definite (as a quadratic form on $N$).  Let us assume it is positive definite, the negative definite case being similar.

Define $f\in C^\infty(\g_\mu^*\times N)$ by
$$f(\alpha,v):=\bh_\gamma(\alpha\rrestr{\m},v)-\bh_\gamma(0,0)$$
and let $f_\OO$ be its restriction to $\OO\times N$.  Define also
$\phi:\g_\mu^*\times N\rightarrow\g_z^*$ by
$$\phi(\alpha,v):=\J_N(v)-\alpha\rrestr{\g_z},$$
and $\phi_\OO$ its restriction to $\OO\times\ N$.

Now \eqref{e2} and \eqref{e3} are together equivalent to the restriction of $f_\OO$ to $\phi_\OO^{-1}(0)$  having a critical point at the point $(\rho\oplus\J_N(v),v)\in\m^*\oplus\g_z^*\times N = \g_\mu^*\times N$, where $\eta\in\g_z$ is the Lagrange multiplier. Indeed, the latter can be written
\begin{eqnarray*}
D_N(f - \left<\phi,\,\eta\right>)&=&0\\
D_{\g_\mu^*}(f - \left<\phi,\,\eta\right>)\rrestr{T_\alpha\OO}&=&0.
\end{eqnarray*}
where $\alpha=\rho+\J_N(v)$ (that is, $\phi(\alpha,v)=0$).  The first of these is equivalent to \eqref{e2}, while the second becomes \eqref{e3}, since $(f - \left<\phi,\,\eta\right>)$ is independent of the $\g_z^*$ component of $\alpha$.

\paragraph{Step 3. Existence of critical points}\\
We claim that for sufficiently small $\OO$, $f_\OO$ is proper.  To see this, we can apply the Splitting Lemma from singularity theory, see for example \cite{PoSt,Mon-SBC}: since $D_Nf(z)=0$ and $D_N^2f(z)$ is non-degenerate, there is a neighbourhood $U$ of $(0,0)$ in $\g_\mu^*\times N$ and change of coordinates on $U$ of the form $(\alpha,v)\mapsto(\alpha,V)$, where $V=V(\alpha,v)$, and a smooth function $g:\g_\mu^*\to\R$ such that
$$f(\alpha,v) = g(\alpha) + f(0,0) + D_N^2f(V,V),$$
and similarly for $f_\OO$.  For simplicity, let us assume $f(0,0)=0$. Now let $I=[a,b]\subset\R$ be a compact interval. Then, since $\OO$ is compact,
$$f_\OO(\alpha,v)\in I \;\Rightarrow\; D_N^2f(V,V) \in [a-\max g(\OO),\,b-\min g(\OO)].$$
Since $D_N^2f(V,V)$ is positive definite, this implies $V$ in contained in a compact set $N_I$.  Then
$$f_\OO^{-1}(I) \subset \OO\times N_I$$
which is compact, and hence indeed $f_\OO$ is proper.  It is also clear that $f_\OO$ is bounded below,  by $\min g(\OO)$.

To finish we use a well-known variational argument: any continuous function $f:X\to\R$ ($X$ a topological space) that is proper and bounded below attains its greatest lower bound, which is therefore a minimum and hence a critical point.  To see this, let $r_1$ be the greatest lower bound of $f(X)$, and let $r_2>r_1$.  Then $K:=f^{-1}([r_1,r_2])$ is non-empty and compact (as $f$ is proper). Therefore $f(K)$ is a compact subset of $\R$ so must be equal to $[r_1,r_2]$, and hence the greatest lower bound $r_1$ is attained by $f$, as claimed.

Applying this to the restriction of $f_\OO$ to $\phi_\OO^{-1}(0)\subset \OO\times N$, for $\alpha = \bar\rho+\bar\epsilon$ sufficiently small that $\OO\subset U$, shows that indeed this restriction of $f_\OO$ has a critical point as required.
\end{proof}

\begin{rem} \label{rem:G-open}
There are two differences with the persistence theorem in \cite{MoTo03}.  The first is that in \cite{MoTo03} it is only assumed that the relative equilibrium is `extremal', rather than formally stable. In practice, formal stability is the most straightforward way of testing for extremality, although it is of course a stronger condition.   An example of an extremal relative equilibrium which is not formally stable is the famous Thompson heptagon in the planar point vortex problem, see the proof in \cite{MoTo13} of its stability  (Corollary 14.7 and the calculation in Sec.\,14.3.5.2).

The second difference is more subtle.  The conclusion in \cite{MoTo03} is that relative equilibria exist for all $\mu'$ in a neighbourhood of $\J(z)$ in $\J(U_0)$ (where, as in the proof above, $U_0$ is the $G$-invariant neighbourhood on which the MGS normal form is valid).  Here on the other hand, existence is only guaranteed for $\mu'\in\J(U)$, the image under $\J$ in $\J(U_0)$ of some (small) $G$-invariant open set $U$.  This image $\J(U)$ may not in general be a full neighbourhood of $\mu=\J(z)$ in $\J(U_0)$.  On the other hand, if $G_\mu$ is compact, then in \cite{MoTo03} it is proved that $\J$ is $G$-open, meaning that the image of $U$ will be open in $\J(U_0)$ and the results will coincide (in this respect).
\end{rem}

\section{Bifurcations}\label{sec bifurcations}
We now consider the problem of bifurcations of branches of relative
equilibria. Starting with a parametrized branch of relative
equilibria of the same symplectic type, we will give sufficient conditions for
the existence of  bifurcating branches.
At the end of this section we study the possibility of bifurcations
from persisting branches consisting of formally stable relative
equilibria. Recall that we assume $G$ acts properly on the symplectic manifold $\PP$ (the phase space) with momentum map $\J:\PP\to\g^*$, and this is equivariant with respect to a (possibly modified)
 coadjoint action. Moreover, $h:\PP\to\R$ is a $G$-invariant Hamiltonian function.

\subsection{Bifurcations from branches of constant symplectic type}\label{subsec bif same type}
\ \\
The following result gives sufficient conditions for the occurrence of
bifurcations from a parametrized branch of relative equilibria of the same
symplectic type.

\begin{thm}\label{thm bifurcations} Let $z\in\PP$ be a relative equilibrium with momentum $\J(z)=\mu$ and velocity
$\xi\in\g_\mu$ written $\xi=\xi^\perp+\eta\in\m\oplus\g_z$ according to
\eqref{gmusplitting}. For a subgroup $K\leq (G_z)_\eta$ and a vector
subspace $W\subset \m^*\times \g_z^K$ let $\bar z:W\rightarrow \PP$,
be a parametrized branch of relative equilibria of the same symplectic type,
satisfying $\bar z(0,0)=z$, $\J(\bar z(\rho,\eta'))=\mu+\rho$,
$G_{\J(\bar z(\rho,\eta'))}=G_\mu$ and $G_{\bar z(\rho,\eta')}=G_z$. Let
$\xi(\rho,\eta')$ be a family of velocities for points of this
branch chosen such that $\Proj_{\g_z}\xi(\rho,\eta')=\eta+\eta'$ (this is always possible). Now suppose that there exists a subgroup $L\leq K$ satisfying
\begin{itemize}
   \item[(i)] $\g_z^L$ is co-central in $\g_\mu^L$ (see Lemma\,\ref{lemma:cocentral}),
      \item[(ii)] $N_{G_z}(L)$ acts on $\Null:=\ker \ed^2_zh_{\xi}\rrestr{N^L}$ with cohomogeneity one,   and
   \item[(iii)] the eigenvalue $\sigma(\rho,\eta')$, with $(\rho,\eta')\in W$, of $\ed^2_{\bar z(\rho,\eta')}h_{\xi(\rho,\eta')}\rrestr{\Null}$ crosses $0$ at $0\in W$.
\end{itemize}
Then, for every  $v\in \Null$ close enough to the origin,  there is a relative equilibrium
$\zt_v$ near $z$ with velocity $\xi_v\in \g_\mu$ close to $\xi$ and
not in the $G$-orbit of any point of the original branch $\bar z(\rho,\eta')$.
The stabilizer of $\zt_v$ is
$$G_{\zt_v}=(G_z)_v\geq L.$$
\end{thm}

\begin{proof}
In the local model, the given branch $\bar z(\rho,\eta')$ can be written as
$$\bar z(\rho,\eta)=[g(\rho,\eta'),\bar\rho(\rho,\eta'),v(\rho,\eta')],$$
with
$[g(0,0),\bar\rho(0,0),v(0,0)]=[e,0,0]$, for
$(\rho,\eta')\in{\m^*}\times {\g_z}^K$. Since relative equilibria
come in group orbits, we can choose in our analysis
$g(\rho,\eta')=e$ without loss of generality.

Since points in the branch satisfy $\J(\bar z(\rho,\eta'))=\mu+\bar\rho(\rho,\eta')+\J_N(v(\rho,\eta'))=\mu+\rho$ it follows that
$$\bar\rho(\rho,\eta')=\rho\quad \text{and}\quad G_{\mu+\rho}=G_\mu$$
Since $\bar z(\rho,\eta')$ and $z$ are of the same symplectic type, we can assume that $G_{\bar z(\rho,\eta')}=G_z$ and in particular
$$(\rho,v(\rho,\eta'))\in {\m^*}^{G_z}\times N^{G_z}.$$
So actually $W\subset {\m^*}^{G_z}\times {\g_z}^{K}\subset {\m^*}^{(G_z)_\eta}\times {\g_z}^{K}$.
Let $\varphi_{\rho,\eta'}: N\rightarrow N$ be the $W$-parametrized
family of diffeomorphisms of $N$ given by
$$\varphi_{\rho,\eta'}(v)=v+v(\rho,\eta').$$
These maps are obviously $K$-equivariant. Notice also that $\J_N(\varphi_{\rho,\eta'}(v))=\J_N(v)$ since $v(\rho,\eta')\in N^{G_z}$.

Notice that $\ed h_{\xi(\rho,\eta')}(\bar z(\rho,\eta'))=0$ if and only
if $\ed \bh_{\eta+\eta'}(\rho,v(\rho,\eta'))=0$ which, putting
$\bh'_{\eta+\eta'}(\rho,v)=\bh(\rho,\varphi_{\rho,\eta'}(v))-\J_N^{\eta+\eta'}(\varphi_{\rho,\eta'}(v))$,
 is in turn
equivalent to $$D_N\bh'_{\eta+\eta'}(\rho,0)=0.$$

Let us call $f(\eta',\rho,v)=\bh'_{\eta+\eta'}(\rho,v)$. The function
$f\in C^\infty(W\times N)$ is by construction $(G_z)_\eta$-invariant.
We have according to Lemma \ref{techlemma} (using $\bh'$ as the
function $\bh$ in the statement of the lemma) that, for every
$(\rho,\eta')\in W$,
$$D_Nf(\eta',\rho,0)=0,$$
and that for any $L\leq K$, the eigenvalues (and their
 multiplicities) of
$$D^2_Nf(\eta',\rho,0)\rrestr{N^L}$$
and
$$\ed^2_{\bar z(\rho,\eta')}h_{\xi(\rho,\eta')}\rrestr{N^L}$$
coincide. In particular, we also have
$$\ker D^2_Nf(0,0,0)\rrestr{N^L}=\ker D_N^2
\bh_{\eta}(0,0)\rrestr{N^L}=\ker\ed^2_{z}h_{\xi}\rrestr{N^L}.$$

We have that $f$ satisfies the hypotheses of Lemma \ref{lemmaD} with
$K=(G_z)_\eta$. Therefore we can conclude that for each
$v\in\ker D^2_N \bh_{\eta}(0,0)\rrestr{N^L}\backslash\{0\}$ there is
a pair $(\rho(v),\eta'(v))\in W$ satisfying
$$D_N
\bh_{\eta+\eta'(v)}(\rho(v),v)=0.$$
We have shown that the point
$[e,\rho(v),v]$ satisfies  \eqref{req2}. We now show that it
also satisfies \eqref{req1}. The argument is again very similar to the one
used in the proof of Theorem \ref{thm persistence nondeg}. Using the
splitting $\g_\mu^*=\m^*\oplus\g_z$ we have that
$\rho(v)+\J_N(v)\in\g_\mu^*$. In fact, since
$(\rho(v),v)\in{{\m}^*}^{G_z}\times N^L$, we can say that
$\rho(v)+\J_N(v)\in{\g_\mu^*}^L$. Calling
$\gamma=D_{\m^*}\bh(\rho(v),v)+\eta+\eta'$, if we show that
$\gamma\in\g_\mu^L$ then \eqref{req2} will follow from the
hypothesis $(i)$. It is clear that
$D_{\m^*}\bh(\rho(v),v)+\eta+\eta'\in\m\oplus\g_z=\g_\mu$, so it only
remains to show that this element is invariant by $L$.
This is true since, on the one hand $\eta+\eta'\in\g_z^L$ and on the other hand, using the Principle of Symmetric Criticality, an argument similar to the one used in Theorem \ref{thm persistence nondeg} shows that $D_{\m^*}\bh(\rho(v),v)\in\m^L$ since $\rho(v)\in  {\m^*}^{G_z}\subset{\m^*}^L$.

Therefore for each $v\in\ker \Null \backslash \{0\}$ close enough to the origin we have that $\zt(v)=\varphi([e,\rho(v),v])$ is a relative equilibrium with velocity $D_{\m^*}\bh(\rho(v),v)+\eta+\eta'\in\g_\mu$. In addition, since $G_{\zt(v)}=(G_z)_{\rho(v)}\cap (G_z)_v$ and we found that $W\in {\m^*}^{G_z}\times \g_z^L$ then $(G_z)_{\rho(v)}=G_z$ and therefore $G_{\zt(v)}= (G_z)_v$. Since $v\in N^L$ we also have $L\leq G_{\zt(v)}$.
\end{proof}

\subsection{Bifurcations from a branch of formally stable points of constant symplectic type}\label{subsec bif formal}
\ \\
As a consequence of Theorem \ref{thm bifurcations} we now show that
branches of formally stable relative equilibria of the same symplectic type
with 2-dimensional fixed point subspaces and satisfying the usual
group-theoretic conditions used in this article, typically have bifurcating
solutions at each point of the branch. This is a phenomenon due exclusively to
the existence of continuous isotropy and cannot happen for branches
of formally stable relative equilibria with discrete stabilizers, for which
the existing persisting branch of relative equilibria of the same
symplectic type must persist without bifurcation.

\begin{thm}\label{2dformallystable}
Assume the hypotheses of Theorem \ref{thm bifurcations} for a branch
$\bar z(\rho,\eta')$ consisting of formally stable points, where
conditions $(ii),(iii)$ are replaced with
\begin{itemize}
\item[(ii)] $\dim N^L=2$, and
\item[(iii)] $\ed^2_{z}h_{\xi}\rrestr{N^L}$ is
 definite.
\end{itemize}
Then, generically, for every $v\in \ker
\ed^2_zh_{\xi}\rrestr{N^L}\backslash \{0\}$ close enough to the
origin,  there is a relative equilibrium $$\zt_v$$ near $z$ with
velocity of the form $\xi_v\in \g_\mu$ close to $\xi$ and not
in the $G$-orbit of any point of the original branch $\bar z(\rho,\eta')$.
\end{thm}

\begin{proof}
Using the function $f\in C^\infty(W\times N)$ introduced in the
proof of Theorem \ref{thm bifurcations} we have that in the local
model we can assume that
$\bar z(\rho,\eta')=[e,\rho,0]$ and that each of these points is a
relative equilibrium with velocity $\xi^\perp+\eta+\eta'$.
 Furthermore,
the signatures of $D_N^2f(\eta',\rho,0)\rrestr{N^L}$ and
$\ed_{\bar z(\rho,\eta')}^2h_{\xi(\rho,\eta')}\rrestr{N^L}$ are the same.
Since $\ed^2_{z}h_{\xi}\rrestr{N^L}$ is definite then so is $D_N^2f(\eta',\rho,0)\rrestr{N^L}$ in a neighbourhood of $(0,0)\in
W$. We therefore
have a map $\psi:W\rightarrow \mathrm{Sym}(2)$ from a neighbourhood
of $(0,0)$ in $W$ to the space of symmetric $2\times 2$ matrices and
its image lies in the set of definite matrices.

If we give coordinates $\{x,y,z\}$ on $\mathrm{Sym}(2)$, where $x,y$
are the diagonal elements and $z$ is off-diagonal, the set of
definite matrices is given by   the two connected
components of the set $\mathrm{Sym}(2,+)$ defined by $xy-z^2>0$. These connected components are the set of positive definite and negative definite matrices respectively. The
set $\mathrm{Sym}(2,-)$ of symmetric $2\times 2$ matrices with index 1
is defined by $xy-z^2<0$. These two regions are separated by the set
$\mathrm{Sym}(2,0)$ of matrices with at least one zero eigenvalue,
given by $xy-z^2=0$. The origin in $\R^3$
represents the zero matrix. Therefore any point in
$\mathrm{Sym}(2,0)$ other than the origin represents a matrix with only one zero
eigenvalue. Suppose that there is an element $\bar\eta\in K$
such that $D^2\J^{\bar\eta}_N(0)\rrestr{N^L}$ is not proportional
to $D_N^2f(\eta',\rho,0)\rrestr{N^L}$. This implies that we can find
some $t\in\R$ satisfying
$D_N^2f(0,0)\rrestr{N^L}+tD^2\J^{\bar\eta}_N(0,0)\rrestr{N^L}\in
\mathrm{Sym}(2,0)\backslash\{(0,0,0)\}$. Notice that this last
matrix is equal to $D_N^2\bar f(0,0)\rrestr{N^L}$ where $\bar
f(\eta',\rho,v)=\bh_{\eta+t\bar\eta+\eta'}(\rho,v)$. Notice also that we can rewrite our starting branch of relative equilibria as $\bar{\bar{z}}(\rho,\eta')=\bar z(\rho,t\bar \eta+\eta')$ but now this produces the map $\bar\psi:W\rightarrow \mathrm{Sym}(2)$ given by $\bar \psi(\rho,\eta')= D_N^2\bar f(\eta',\rho,0)\rrestr{N^L}$
which satisfies
$\bar\psi(0,0)\in\mathrm{Sym}(2,0)\backslash\{(0,0,0)\}$ and generically will cross this
surface. In other words, substituting the initial
velocity of $z$, the element $\xi$, by the also admissible velocity
$\xi^\perp+\eta+t\bar\eta$, we can force all the conditions of Theorem \ref{thm
bifurcations} to be satisfied generically, and the existence of the
stated bifurcating branch follows.
\end{proof}
\section{Example. The sleeping Lagrange top}\label{sec example}
This classic example will allow us to illustrate our results on the
persistence, stability and bifurcations of  relative
equilibria.
We will briefly summarize the construction of this system, referring
to \cite{OrRa99} for more details, since we will be using the same
notations and conventions.

The Lagrange top is a mechanical system consisting in a rigid body with an axis of symmetry and a fixed point in the presence of an homogeneous vertical gravitational field. Relative equilibria are steady rotations around the vertical axis and the axis of symmetry. Mathematically, this is the symmetric Hamiltonian system
$$(T^*\mathrm{SO}(3),\w_c,\mathbb{T}^2,\J,h),$$
where $\w_c$ is the canonical symplectic form on the cotangent bundle $T^*\mathrm{SO}(3)$ and $\J:T^*\mathrm{SO}(3)\rightarrow\mathfrak{t}^2$ its canonical momentum map. We will use the right trivialization of $T^*\mathrm{SO}(3)$, and will identify $\mathfrak{so}(3)$ with $\R^3$ and $\mathfrak{t}^2$ with $\R^2$ in the usual ways. The Euclidean structures on these two vector spaces will allow us to identify them also with the duals of their Lie algebras. In this setup, a point in $T^*\mathrm{SO}(3)$ is represented by a pair $(\Lambda, \mathbf{\pi})$, where $\Lambda\in\mathrm{SO}(3)$, viewed as a matrix group, and $\pi\in\R^3$. The several relevant elements for the analysis are:
\begin{itemize}
\item The action $\mathbb{T}^2\times T^*\mathrm{SO}(3)\rightarrow T^*\mathrm{SO}(3)$, given by
\begin{equation}\label{top action}((\theta_1,\theta_2),(\Lambda,\mathbf{\pi}))\mapsto (e^{\theta_1\widehat{\mathbf{e_3}}}\Lambda e^{-\theta_2\widehat{\mathbf{e_3}}},e^{\theta_1\widehat{\mathbf{e_3}}}\mathbf{\pi}).\end{equation}

\item The Hamiltonian
$$h(\Lambda,\mathbf{\pi})=mgl\mathbf{e_3}\cdot\Lambda\mathbf{e_3}+\frac 12 \mathbf{\pi}\cdot\mathbb{I}^{-1}_{\Lambda}\mathbf{\pi},$$
where $\mathbb{I}=\operatorname{diag} (I_1,I_1,I_3)$ and $\mathbb{I}_\Lambda=\Lambda\mathbb{I}\Lambda^{-1}$. Here $I_1,I_3$ are the principal moments of inertia of the body, $I_3$ corresponding to the axis of symmetry.

\item The momentum map $$\J(\Lambda,\mathbf{\pi})=(\mathbf{\pi}\cdot\mathbf{e_3},-\mathbf{\pi}\cdot\Lambda\mathbf{e_3})\in\R^2\simeq {\mathfrak{t}^2}^*.$$
\end{itemize}
The sleeping Lagrange top is a family of relative equilibria of this Hamiltonian system, consisting in the top rotating around the vertical position, with the gravity and symmetry axes aligned. This corresponds to the points in phase space of the form $\mathbb{T}^2\cdot z$, where $z=(\Lambda,\mathbf{\pi})=(I,\lambda I_3\mathbf{e_3})$, $I$ is the identity $3\times 3$ matrix and $\lambda$ is any non-zero real number, corresponding to the angular velocity of the rotation. It is easy to see from \eqref{top action} that the stabilizer of any sleeping Lagrange top is
$$G_z=\{(\theta_1,\theta_2)\in\mathbb{T}^2\,:\, \theta_1=\theta_2\}\simeq S^1.$$
The Lie algebra $\g_z$ can be identified with $\R$ as
$$\g_z=\{(\eta,\eta)\in \R^2\,:\,\eta\in \R\}.$$
Therefore this relative equilibrium has continuous isotropy and its velocity is not uniquely determined. It is easy to see that a possible orthogonal velocity is given by $\xi^\perp=\frac 12(\lambda,-\lambda)\in\R^2$, corresponding to the choice
\begin{equation}\label{subalgebram}
\m=\{(\gamma,-\gamma)\in \R^2\,:\,\gamma\in \R\}.
\end{equation}

 It follows that the set of all possible velocities for this relative equilibrium is $(\eta+\lambda/2,\,\eta-\lambda/2)$, for any real number $\eta$. Its momentum value is
 \begin{equation}\label{momentum ltop}\mu=(\lambda I_3,-\lambda I_3).
 \end{equation}
 Since the symmetry group is Abelian, $G_\mu=G=\mathbb{T}^2$, and $\dim \m^*=1$.

It follows  that with respect to the basis of $N$ given by
$$v_1=(\mathbf{e_1},0),\,v_2=(\mathbf{e_2},0),\,v_3=(0,\mathbf{e_1}),\,v_4=(0,\mathbf{e_2})$$
the symplectic normal space $N$ at $z=(I,\lambda I_3\mathbf{e_3})$ can be identified with $\R^4$, and
\begin{equation}\label{hesslagrange}\ed_z^2h_{(\eta+\lambda/2,\eta-\lambda/2)}\rrestr{N}=\left(\begin{array}{cccc} A & 0 & 0 & B\\ 0 & A & -B & 0\\ 0  & -B & C & 0 \\ B & 0 & 0 & C\end{array}\right)\end{equation}
where
$$A=-mgl -\frac{\lambda^2 I_3}{2I_1}\left(2I_3-I_1\right)-\eta\lambda I_3,\quad B=\frac{\lambda}{2I_1}\left(2I_3-I_1\right)-\eta,\quad C=\frac{1}{I_1}.$$

\paragraph{Stability.} We will fix $\lambda\neq 0$ and study the nonlinear stability of $z=(I,\lambda I_3)$. In order to apply Theorem \ref{stability} we need to find an admissible velocity $\xi$ for $z$ such that $\ed_z^2h_{\xi}\rrestr{N}$ is definite. This velocity will be of the form $\xi=(\eta+\lambda/2,\eta-\lambda/2)$.

The eigenvalues of $\ed_z^2h_{(\eta+\lambda/2,\eta-\lambda/2)}\rrestr{N}$ are given by
\begin{equation}\label{eigenvalues sleeping}\sigma_{\pm}=\frac 12\left( (A+C)\pm\sqrt{(A+C)^2-4(AC-B^2)}  \right),\end{equation} each with multiplicity 2. It follows that
this matrix is definite provided $AC-B^2> 0$. Since $\eta$ is arbitrary and $\lambda\neq 0$, we will make $\eta=k\lambda$ for
arbitrary $k\in \R$. Then we have
$$AC-B^2=\frac{1}{4I_1}\left(\lambda^2\left(I_3(4k+2)-I_1(4k^2+4k+1)  \right)-4glm \right).$$
Then, the relative equilibrium will be stable if we can find
some $k\in \R$ for which
$\lambda^2>\frac{4glm}{I_3(4k+2)-I_1(4k^2+4k+1)}$. It is clear that this
is always possible if $\lambda^2$ is sufficiently large, so we are
interested in the element of the family of sleeping Lagrange tops
with minimum $\vert \lambda \vert$ yet nonlinearly stable. For that,
we find a minimum of the above function of $k$, which happens when
$k=\frac{I_3-I_1}{2I_1}$, and then, the relative equilibrium
$z=(I,\lambda I_3\mathbf{e_3})$ is nonlinearly stable provided
$$\lambda^2>\frac{4glmI_1}{I_3^2}$$
which is the well known classical stability condition (fast-top condition).

\paragraph{Persistence.}
It is straightforward to show that the linear action of $G_z\simeq S^1$ on $N$ is given by
\begin{equation}\label{top linear action}\theta\in S^1\mapsto \left(\begin{array}{cccc}
\cos \theta & -\sin \theta & 0 & 0 \\
\sin \theta & \cos \theta & 0 & 0 \\
0 & 0 & \cos \theta & -\sin \theta \\
0 & 0 & \sin \theta & \cos \theta\end{array}\right).\end{equation}

It follows that $N^{G_z}=\{0\}$, whence $\ed^2_zh_\xi\rrestr{N^{G_z}}$ is trivially non-degenerate. Note also that
$G_\mu=\mathbb{T}^2$ is a torus. Recall  that
$(T^*\mathrm{SO}(3))_{(S^1)}$ is the set of points of the form $(I,\lambda I_3\mathbf{e_3})$ with $\lambda\neq 0$, that is, the set of all Lagrange tops which are not at rest. Taking into account expression \eqref{momentum ltop}, it follows from Theorem \ref{thm persistence same orbit type} that near a sleeping Lagrange top with phase space point $(I,\lambda_0 I_3\mathbf{e_3})$ there is $\epsilon>0$ such that for every $r\in (-\epsilon,\epsilon)$ there is a $\mathbb{T}^2$-orbit of  relative equilibria of orbit type $(S^1)$ that forms a persisting branch of nonlinearly stable relative equilibria  of the form $(I,(\lambda_0+r) I_3\mathbf{e_3})$. Furthermore this branch is a local symplectic submanifold of $T^*\mathrm{SO}(3)$.

 Note that in this
particular case, the persistence result already follows from Remark
4.3 in \cite{LeSi98}, although in that reference the conditions are
stronger in general, requiring the relative equilibrium to be
non-degenerate. By the uniqueness
argument of Theorem \ref{thm persistence same orbit type}, it follows that this
persisting branch corresponds exactly to a neighbourhood of $\lambda$
in the family of sleeping Lagrange tops, which exists for any
non-zero $\lambda$, as can be shown by a direct computation. Note
that the case $\lambda=0$ is not contained in the persisting branch
of sleeping Lagrange tops since in this case $G_\mu=\mathbb{T}^2$, and
then the symplectic type of this point is different from the elements of the branch. The case $\lambda=0$ corresponds to the unstable
equilibrium of the top in the upright position.

\paragraph{Bifurcations.}
 We now apply Theorem \ref{thm bifurcations} to study the possible bifurcations of relative equilibria from the family of sleeping Lagrange tops. We will start by a relative equilibrium in the family of sleeping Lagrange tops of the form $z=(I,\lambda_0I_3\mathbf{e_3})$ with  momentum $\mu=(\lambda_0I_3,-\lambda_0I_3)$ and orthogonal velocity $\xi^\perp=(\lambda_0/2,-\lambda_0/2)\in\m$  according to the choice \eqref{subalgebram} . Notice that since $G_z$ is Abelian, $(G_z)_\eta=G_z\simeq S^1$. Let us choose $L=K=\{e\}$ and $W=\m^*\times\{0\}\subset \m^*\times\g_z$.  We can identify
 $$W=\{(\rho,-\rho)\, :\, \rho\in\R\}$$
 and define a local branch of relative equilibria $\bar z(\rho)=(I,(\lambda_0+\rho/I_3)I_3\mathbf{e_3})$. Notice that $\bar z(0)=z$ and $\J(\bar z(\rho))=\mu + (\rho,-\rho)$. It is also clear that $G_{\J(\bar z(\rho))}=\mu + (\rho,-\rho)=G_\mu=G$ since $G$ is Abelian, and also that $G_{\bar z(\rho)}=G_z\simeq S^1$. We choose the family of velocities $\xi(\rho)=((\lambda_0+r)/2,-(\lambda_0+r)/2)+(\eta,\eta)$ for a fixed, although arbitrary, element $\eta\in\R$.

 It is clear, again from the Abelian character of $G$, that condition $(i)$ in Theorem \ref{thm bifurcations} is trivially satisfied. In order to check condition $(iii)$, notice that $N^L=N$. According to \eqref{eigenvalues sleeping} we see that $\sigma_+(\rho)>0$ for every value of $\rho$ but
 $$\mathrm{sign}(\sigma_-(\rho))=\mathrm{sign}( AC-B^2).$$

 Computing the value of this last expression along the branch $\bar z(\rho)$ we get
 \begin{eqnarray*}4I_1I_3^2(AC-B^2) & = & \rho^2(2I_3-I_1)+\rho\left((2\eta+\lambda_0)I_3(I_1-I_3)-2I_3^2\lambda_0\right) \\ & & -
 \left((2\eta+\lambda_0)^2I_1I_3^2-2(2\eta+\lambda_0)I_3^3\lambda_0+4mglI_3^2 \right).\end{eqnarray*}

 It is straightforward to check that $\eta$ can be chosen such that this expression vanishes at $\rho=0$ if one can solve the independent term for $\eta$. This is possible precisely when
 $$\lambda_0^2>\frac{4mglI_1}{I_3^2},$$
 that is, if the local branch $\bar z(\rho)$ is centred at a point in the formally stable range. In order to check that the eigenvalue $\sigma_{-}(\rho)$ actually changes sign at $\rho=0$ we have to guarantee that the coefficient of the linear term is different from zero. But this is true since if both the coefficients of the linear and independent term vanish simultaneously we would have
 $$\left(\frac{2I_3\lambda_0}{I_1-I_3}\right)^2I_1I_3^2-2 \frac{2I_3\lambda_0}{I_1-I_3} I_3^3\lambda_0+4mglI_3^2=0$$
 and the above expression is always positive so $\sigma(\rho)$ changes sign at $\rho=0$ and condition $(iii)$ in Theorem \ref{thm bifurcations} is satisfied.

 Thus it remains to study condition $(ii)$. For that, notice that since $L=\{e\}$ and $N_{S^1}(e)\simeq\mathrm{SO}(2)$ we have to check if
 $$\ker \ed_z^2h_{(\eta+\lambda/2,\eta-\lambda/2)}\rrestr{N},$$
 with $\eta$ chosen in a way that the crossing condition $(iii)$ is satisfied, is two-dimensional and invariant under the action of $G_z\simeq\mathrm{SO}(2)$ given by \eqref{top linear action}. It is easy to see from \eqref{hesslagrange} that for $\lambda=\lambda_0$,
  $$\ker \ed_z^2h_{(\eta+\lambda/2,\eta-\lambda/2)}\rrestr{N}=\mathrm{span}\left\langle\left(a,0,0, -\frac BC a\right),\left(0,b,\frac BC b, 0\right)\right\rangle $$
  with respect to the basis $\{v_1,v_2,v_3,v_4\}$ of $N$. It is clear from \eqref{top linear action} that this space is a $S^1$-module equivariantly isomorphic to $\R^2$ equipped with the standard action of $S^1$. Therefore $(ii)$ is also satisfied. It follows from Theorem \ref{thm bifurcations} that for every $v\in \ker \ed_z^2h_{(\eta+\lambda/2,\eta-\lambda/2)}\rrestr{N}$ near the origin, there is a relative equilibrium for the Lagrange top system not contained in the branch of sleeping tops. Since the action of $S^1$ on this space is free outside the origin, these relative equilibria have trivial isotropy and correspond to precessing tops, where the symmetry and gravity axes are not aligned.

\section{Example. 2 point vortices on the sphere}\label{sec example2}
Consider the system of 2 point vortices on the sphere \cite{KN98, LMR01}, which is a simple system
where every motion is a relative equilibrium.  The phase space is $P=S^2\times S^2\setminus\Delta$,
where $\Delta$ is the diagonal, with symplectic form $\omega=\Gamma_1\omega_1+\Gamma_2\omega_2$,
where $\omega_j$ is the natural $\mathrm{SO}(3)$-invariant symplectic form on the $j^{\mathrm{th}}$
copy of the sphere and $\Gamma_j\in\mathbb{R}$ are the vorticities of the two points (taken to be
non-zero).   The equation of motion is
$$\dot \xx_1 = \Gamma_2\frac{\xx_2\times \xx_1}{1-\xx_1\cdot \xx_2},\qquad
\dot \xx_2 = \Gamma_1\frac{\xx_1\times \xx_2}{1-\xx_1\cdot \xx_2}.$$
This is a Hamiltonian system, with Hamiltonian
$$h(\xx_1, \xx_2) = -\Gamma_1\Gamma_2\log\|\xx_1-\xx_2\|^2$$
where $\|\xx_1-\xx_2\|$ is the Euclidean distance between the points (different versions of this
differ by factors of $2\pi$, but this can always be compensated by rescaling time).  There are two possible $G$-Hamiltonian systems depending on whether $\Gamma_1$ or $\Gamma_2$ are different or equal. If $\Gamma_1\neq\Gamma_2$ then the group $\mathrm{SO}(3)$ acts on $\PP$ by the diagonal action. We think of each copy of
$S^2$ as embedded as the unit sphere in $\mathbb{R}^3$. This action is Hamiltonian with momentum map given by, after identifying $\mathfrak{so}(3)^*$ with $\mathbb{R}^3$ as
usual,
\begin{equation}\label{moment-vort}\mathbf{J}(\xx_1,\xx_2) = \Gamma_1\xx_1+\Gamma_2\xx_2.\end{equation}
If $\Gamma_1=\Gamma_2$ then the system supports a Hamiltonian action of the direct product $\mathrm{SO}(3)\times \mathbb{Z}_2^\tau$ where $\mathrm{SO}(3)$ acts as in the previous case and $\mathbb{Z}_2^\tau$ is the reflection group generated by  $\tau(\xx_1,\xx_2)=(\xx_2,\xx_1)$. The momentum map for this action is also given by the expression \eqref{moment-vort}.

The action of $\mathrm{SO}(3)$ (respectively $\mathrm{SO}(3)\times \mathbb{Z}_2^\tau$)  on  phase space is locally free except where $\xx_1=-\xx_2$ (antipodal
points), in which case the isotropy is $\mathrm{SO}(2)$ (respectively $\widetilde{\mathrm{O}(2)}$, which is generated by the group $\mathrm{SO}(2)$ of rotations around $\mathbf{\xx_1}$ and an element $(R_\mathbf{n},\tau)\in \mathrm{SO}(3)\times \mathbb{Z}_2^\tau$. Here $R_\mathbf{n}$ denotes a  rotation of angle $\pi$ around a vector $\mathbf{n}$ perpendicular to $\xx_1$. Since, for a given $\mathrm{SO}(2)$, the
subspace $\PP^{\mathrm{SO}(2)}=\PP^{\widetilde{\mathrm{O}(2)}}$ consists of just two points, they are necessarily
equilibria for any invariant Hamiltonian. Non-antipodal pairs have stabilizers $\{e\}$ or $\mathbb{Z}_2(\mathbf{n})$ respectively, where $\mathbb{Z}_2(\mathbf{n})$ is the subgroup generated by the element $(R_\mathbf{n},\tau)\in \mathrm{SO}(3)\times \mathbb{Z}_2^\tau$.

With 2 point vortices, as already mentioned, every solution is a relative equilibrium, and indeed
given any $\xx_1,\xx_2$ which are not antipodal, the angular velocity is
$$\xi = \frac 1{2\sin^2(\theta/2)}\mu$$
where $\mu=\mathbf{J}(\xx_1,\xx_2)$ and $\theta$ is the angle subtended by the two points.  With respect to antipodal points, the situation changes depending on the two possibilities for the vorticities.
\paragraph{1. Case $\Gamma_1\neq\Gamma_2$.}
In the
limit of an antipodal pair, so as $\xx_2\to -\xx_1$, one has $\xi\to \frac
12\mu=\frac12(\Gamma_1-\Gamma_2)\xx_1\neq0$. This limiting $\xi$ is of course an element of
$\mathfrak{g}_z$ for $z=(\xx_1,-\xx_1)$, in order that $z$ be an equilibrium point. In this case  $N$ is 2-dimensional.  Specifically, take $z=(\mathbf{e_3},-\mathbf{e_3})\in S^2\times S^2$, and using coordinates
$(x_1,x_2,y_1,y_2)$
the tangent space to the group orbit is spanned by $(1,-1,0,0)$ and $(0,0,1,-1)$. Since
$\dim\mathfrak{g}_z=1$ the rank of $T_z\mathbf{J}$ is also 2, and its kernel is spanned by the
tangent vectors \begin{eqnarray}\label{Nvortices}v_1=(\Gamma_2,-\Gamma_1,0,0),\quad
v_2=(0,0,\Gamma_2,-\Gamma_1).\end{eqnarray} This is therefore also the symplectic normal space $N$
when $\Gamma_1\neq\Gamma_2$.  On this space, and with this basis,
$$\ed^2_zh\rrestr{N} = \frac12\Gamma_1\Gamma_2(\Gamma_1-\Gamma_2)^2I_N,\quad
\ed^2_z\mathbf{J}\rrestr{N}= \Gamma_1\Gamma_2(\Gamma_1-\Gamma_2)I_N,$$
where $I_N$ is the identity matrix on $N$.  Thus, for $\eta\in\mathfrak{g}_z\simeq\R$,
$$\ed^2_zh_\eta\rrestr{N} =
\frac12\Gamma_1\Gamma_2(\Gamma_1-\Gamma_2)(\Gamma_1-\Gamma_2-2\eta)I_N.$$
Thus, assuming $\Gamma_1\neq\Gamma_2$, this is degenerate precisely at $\eta =
\frac12(\Gamma_1-\Gamma_2)$; note that this is precisely the limiting velocity $\frac12\mu$ mentioned
above. In particular, note that $\mu=(\Gamma_1-\Gamma_2)\mathbf{e_3}\neq 0$ and then
$G_\mu=G_z=\mathrm{SO}(2)$ and $\m^*=0$.

With reference to Theorem \ref{thm persistence deg}, the relative equilibrium $z$ is an equilibrium,
so $\xi^\perp=0$, and let $\eta=\frac12(\Gamma_1-\Gamma_2)$.
We take $L=\{e\}$ and $K=G_z=\mathrm{SO}(2)$.  Therefore $N^L=N$ and
$[\g_\mu^L,\g_\mu^L]=[\g_\mu,\g_\mu]=0$ since $G_\mu$ is Abelian.
 We have
$$\ed^2_zh_{\eta+\eta'}\rrestr{N} = \Gamma_1\Gamma_2(\Gamma_2-\Gamma_1)\eta'I_N,$$
which has a unique eigenvalue $\lambda(\eta')=\Gamma_1\Gamma_2(\Gamma_2-\Gamma_1)\eta'$ that changes
sign at $\eta'=0$, satisfying hypothesis $(i)$. At $\eta'=0$, $\ker \ed^2_zh_{\eta}\rrestr{N}=N$
which is 2-dimensional, and from \eqref{Nvortices} is a $\mathrm{SO}(2)$-module isomorphic to $\R^2$
equipped with the standard circle action, satisfying  hypothesis $(ii)$. Since the action of $G_z$ is
free outside the origin, it follows from Theorem \ref{thm persistence deg} that for every element of
$N\backslash\{0\}$ there is a relative equilibrium near $z$ with trivial isotropy. That is, a
rotating configuration of non-antipodal pairs.

On the other hand the content of  Theorem \ref{thm persistence nondeg} is empty in this case, since $\m^*=0$ and therefore we only get the same equilibrium point, with different velocities $\eta'\in\g_z$. On the third hand, Theorem \ref{thm persistence same orbit type} is not applicable here since $\PP_{(\mathrm{SO}(2))}$ consists of a single group orbit.

\paragraph{2. Case $\Gamma_1=\Gamma_2$.}
In this case, noting that for $z=(\xx_1,-\xx_1)$ we have $\mu=\J(z)=0$ and we find that $G_\mu=\mathrm{SO}(3)\times \mathbb{Z}_2^\tau$ and $G_z=\widetilde{\mathrm{O}(2)}$. Therefore the kernel of $T_z\J$ coincides with the tangent space to the group orbit at $z$, which implies that $N$ is trivial. However $\m^*=\xx_1^\perp$ and is therefore 2-dimensional. Since $\ed^2_zh_{\eta+\eta'}\rrestr{N}$ is trivially non-degenerate for any $\eta'$, Theorem \ref{thm persistence deg} is not applicable. Also, since in this case $G_z=\widetilde{\mathrm{O}(2)}$, then  $\PP_{\widetilde{\mathrm{O}(2)}}$ consists of a single group orbit and Theorem \ref{thm persistence same orbit type} is not applicable either.

With respect to Theorem \ref{thm persistence nondeg}, let us choose $K=\mathbb{Z}_2(\mathbf{n})$. Then we have ${\m^*}^K=\mathbb{R}\mathbf{n}$ which is Abelian, fulfilling condition $(i)$. We also have $\g_z^K=0$. Since $N^K$ is trivial, condition $(ii)$ is automatically satisfied too. It follows from Theorem \ref{thm persistence nondeg} that there is a smooth branch of ($G$-orbits of) relative equilibria near $z$ with stabilizers containing $\mathbb{Z}_2(\mathbf{n})$. Noticing that for $\Gamma_1=\Gamma_2$ the stabilizer of  non-antipodal  pairs is precisely $\mathbb{Z}_2(\mathbf{n})$, we find that the application of the theorem produces the branch of rotating non-antipodal relative equilibria near antipodal equilibria.

Since every motion is a relative equilibrium, the reduced dynamics for this system is trivial so stability holds trivially.  Moreover, the theorems of Section\,\ref{sec bifurcations} do not apply as there are no non-trivial branches of relative equilibria from which to bifurcate.

\end{document}